\documentclass[letterpaper, 11pt,  reqno]{amsart}

\usepackage{amsmath,amssymb,amscd,amsthm,amsxtra,esint}
\usepackage{hyperref}

\usepackage[all]{xy}

\allowdisplaybreaks[2]

\newtheorem{theorem}{Theorem} [section]

\newtheorem{lemma}[theorem]{Lemma}
\newtheorem{proposition}[theorem]{Proposition}
\newtheorem{remark}[theorem]{Remark}

\newtheorem{definition}[theorem]{Definition}
\newtheorem{corollary}[theorem]{Corollary}

\renewcommand{\1}{\hspace{0.5mm}\text{I}\hspace{0.5mm}}
\renewcommand{\2}{\text{I \hspace{-2.8mm} I} }

\newcommand{\noi}{\noindent}
\newcommand{\Z}{\mathbb{Z}}
\newcommand{\R}{\mathbb{R}}
\newcommand{\C}{\mathbb{C}}
\newcommand{\T}{\mathbb{T}}

\let\Re=\undefined\DeclareMathOperator*{\Re}{Re}
\let\Im=\undefined\DeclareMathOperator*{\Im}{Im}

\let\P= \undefined
\newcommand{\P}{\mathbf{P}}

\newcommand{\Q}{\mathcal{Q}}

\newcommand{\N}{\mathcal{N}}
\newcommand{\I}{\mathcal{I}}

\newcommand{\RR}{\mathcal{R}}

\newcommand{\F}{\mathcal{F}}

\newcommand{\al}{\alpha}

\newcommand{\dl}{\delta}

\newcommand{\eps}{\varepsilon}

\newcommand{\g}{\gamma}
\newcommand{\G}{\Gamma}
\newcommand{\ld}{\lambda}

\newcommand{\s}{\sigma}

\newcommand{\ft}{\widehat}

\newcommand{\wt}{\widetilde}

\newcommand{\dx}{\partial_x}
\newcommand{\dt}{\partial_t}

\renewcommand{\l}{\ell}

\newcommand{\les}{\lesssim}

\renewcommand{\S}{\mathcal{S}}

\newcommand{\B}{\mathcal{B}}

\newtheorem*{ackno}{Acknowledgements}

\newcommand{\bk}{{\bf k}}

\newcommand{\J}{\mathcal{J}}

\numberwithin{equation}{section}
\numberwithin{theorem}{section}

\newcommand{\GG}{\mathcal{G}}

\newcommand{\Ind}{\mathrm{Ind}}

\begin{document}
\baselineskip = 14pt

\title[quadratic derivative NLS on the circle]
{Normal form approach to global well-posedness of the quadratic derivative 
nonlinear Schr\"odinger equation on the circle}

\author[J.~Chung, Z.~Guo, S.~Kwon, and T.~Oh]{Jaywan Chung, Zihua Guo, Soonsik Kwon, and Tadahiro Oh}

\address{Jaywan Chung\\
National Institute for Mathematical Sciences\\
70, Yuseong-daero 1689 beon-gil, Yuseong-gu,  Daejeon, 34047, Korea
}

\email{jchung@nims.re.kr}

\address{Zihua Guo\\
School of Mathematical Sciences\\
Monash University\\
VIC 3800, Australia, 
and 
LMAM, School of Mathematical Sciences, Peking University, Beijing 100871, China} 

\email{zihua.guo@monash.edu}

\address{Soonsik Kwon\\
 Department of Mathematical Sciences, Korea Advanced Institute of Science and Technology, 
 291 Daehak-ro, Yuseong-gu,
 Daejeon, 34141, Korea}
\email{soonsikk@kaist.edu}

\address{
Tadahiro Oh, School of Mathematics\\
The University of Edinburgh\\
and The Maxwell Institute for the Mathematical Sciences\\
James Clerk Maxwell Building\\
The King's Buildings\\
Peter Guthrie Tait Road\\
Edinburgh\\ 
EH9 3FD\\
 United Kingdom}

\email{hiro.oh@ed.ac.uk}

\keywords{quadratic derivative nonlinear Schr\"odinger equation; normal form; Cole-Hopf transform;
 Fourier-Lebesgue space; well-posedness; finite time blowup}
\subjclass[2010]{35Q55}

\begin{abstract}
We consider  the quadratic derivative nonlinear Schr\"odinger equation (dNLS)
on the circle.
In particular, we develop an infinite iteration scheme of normal form reductions for dNLS.
By combining this normal form procedure with the Cole-Hopf transformation, 
we prove unconditional global well-posedness  in $L^2(\T)$, 
and more generally 
 in certain Fourier-Lebesgue spaces $\F L^{s, p}(\T)$,  under the 
mean-zero and smallness assumptions.
As a byproduct, we construct an infinite sequence of
quantities that are invariant under the dynamics.
We also show the necessity of the smallness assumption
by explicitly constructing a finite time blowup solution
with non-small mean-zero initial data.

\end{abstract}


\maketitle

\date{\today}

\smallskip

\section{Introduction}
\label{SEC:intro}

In this paper, we consider the Cauchy problem for the following 
quadratic derivative nonlinear Schr\"odinger equation (dNLS) posed on the circle  $\T = \R/(2\pi \Z)$: 
\begin{equation}\label{DNLS1}
 \begin{cases}
i \dt u + \dx^2 u  =  \ld u \dx u \\
u|_{t = 0} = \phi,
\end{cases}
\quad 
(t,x) \in \R \times \T,
\end{equation}

\noi
where $u(t, x)$ is a complex-valued function
and $\ld \in \C\setminus \{0\}$.
Our main goal is to study the global-in-time behavior of solutions to \eqref{DNLS1}.
By the transformation $u \mapsto \frac{\ld}{i} u$, we may assume  $\ld = i$ in \eqref{DNLS1}.
Therefore, in the remaining part of this paper, 
we consider the following Cauchy problem:
\begin{equation}
\begin{cases}
 \dt u = i \dx^2 u + u \dx u\\
u|_{t = 0} = \phi.
\end{cases}
\label{DNLS1a}
 \end{equation}

Let us first compare some basic properties
of \eqref{DNLS1a}
against   the usual NLS with a power-type nonlinearity:
\begin{align}
\label{NLS1}
\dt u  = i \dx^2 u  \pm i |u|^{p-1}u
\end{align}

\noi
and  the cubic derivative NLS (DNLS):
\begin{equation}\label{DNLS2}
   \dt u  = i  \dx^2 u  + \dx (|u|^2u).
\end{equation}

\noi
It is well known that \eqref{NLS1} and \eqref{DNLS2}
are invariant under modulations:
$u \mapsto e^{i \theta} u$, 
which leads to the conservation of mass (= the $L^2$-norm).
On the other hand, 
the equation  \eqref{DNLS1a} is not  invariant under 
such modulations.
Another difference from \eqref{NLS1} and \eqref{DNLS2}
is the lack of conservation laws.
While \eqref{NLS1} and \eqref{DNLS2}
enjoy the conservation of mass, momentum, and energy, 
the equation
\eqref{DNLS1a} does not possess 
any  `natural' conservation law, 
except
for the conservation of the spatial mean:
\begin{align}
\label{mean}
 \int_\T u(t,x)dx  =\int_\T \phi(x) dx. 
 \end{align}

The Cauchy  problems for 
the usual NLS  \eqref{NLS1} and the cubic DNLS \eqref{DNLS2}
have been studied extensively
by many mathematicians 
\cite{GV2, Tsutsumi,  CazW2, Bo2, 
Takaoka, Herr,
 HTT11,   W}
and they are known to be locally/globally well-posed 
in Sobolev spaces $H^s$ for some range of $s$.
Moreover, the solution map: $u(0) \in H^s \mapsto u(t) \in H^s$
is known to be uniformly continuous on bounded sets.
The situation for 
dNLS  \eqref{DNLS1a} 
is entirely different.
Indeed, Christ \cite{Christ} proved
that \eqref{DNLS1a} on $M = \T$ or $\R$ is ill-posed in $H^s(M)$
for any $s \in \R$ 
by exhibiting the following  norm inflation phenomenon;
given any $\eps > 0$, 
there exist a solution $u$ to \eqref{DNLS1a} on $M$
and $t_\eps  \in (0, \eps) $ such that 
\begin{align}
 \| u(0)\|_{H^s(M)} < \eps \qquad \text{ and } \qquad \| u(t_\eps)\|_{H^s(M)} > \eps^{-1}.
\label{NI}
 \end{align}

\noi
This in particular implies the failure
of continuity of the solution map (at the trivial function)
and hence the ill-posedness of \eqref{DNLS1a}.
See also 
\cite{TAO} (and \cite{KTz, Tzv} in the context of the related Benjamin-Ono equation)
for the failure of local uniform continuity
of the solution map on $\R$.

In view of the above ill-posedness result \cite{Christ}, 
we can not expect to have well-posedness for \eqref{DNLS2} in the usual Sobolev spaces.
In the non-periodic setting, however, 
there are several local well-posedness results
for \eqref{DNLS1a}, either in weighted Sobolev spaces
or by adding an extra condition on initial data.
Kenig-Ponce-Vega \cite{KPV}
and Chihara \cite{Chihara} established local well-posedness of  \eqref{DNLS1a} on $\R$
in weighted Sobolev spaces with sufficiently high regularity.
In \cite{Stef},  Stefanov 
proved local well-posedness of \eqref{DNLS1a}
in $H^1(\R)$ with the small `disturbance' condition:
 $\sup_{x\in\R} \big|\int_{-\infty}^x \phi(y)\,dy\big| <  \eps$
 for some $\eps > 0$. 
Note that these results make use of 
the strong dispersive effect on $\R$
such as  local smoothing estimates
that are not available on the circle. 
In fact, there seems to be no known well-posedness result of \eqref{DNLS1a}
on $\T$.

Before we state our main results, 
let us define 
the homogeneous Fourier-Lebesgue space $ {\F L}_0^{s, p}(\T)$
for {\it mean-zero} functions on $\T$
 via the norm:
\[
\| \phi \|_{\F L_0^{s, p}(\T)} = \big\| |k|^{s} \ft \phi(k) \big\|_{\l^p(\Z_0)},
\]

\noi
where  $\Z_0 := \Z\setminus \{0\}$.
Note that, if we assume that an initial condition $\phi$ has mean 0, 
then the corresponding solution $u(t)$ to \eqref{DNLS1a} defined 
on a time interval $I$ also has mean 0 for each $t \in I$
thanks to  the conservation of the spatial mean \eqref{mean}.

We now state our main results.

\begin{theorem}[small data global well-posedness
for mean-zero initial data]\label{THM:main}
Suppose that  $(s, p)$ satisfy  \textup{(i)} $s>\frac 12 -\frac{1}{p}, \  p> 2$ or 
\textup{(ii)} $s\geq 0, \ p=2$. 
Then, there exists  $ \dl_0 = \dl_0(s, p) >0$ such that
dNLS \eqref{DNLS1a} is globally well-posed
in $\F L_0^{s, p}(\T)
\cap \{ 
\|\phi \|_{\F L_0^{s,p}(\T)} \leq \dl_0\}$.
Moreover, the uniqueness holds unconditionally.
Namely, the uniqueness holds in the entire $C(\R; \F L_0^{s, p}(\T))$.

Given a solution $u$ to \eqref{DNLS1a} constructed above, 
 there exists an infinite sequence $\{\Q_k (t) \}_{k \in \Z_0}$
such that $\Q_k = \Q_k[u]: \R \to \C$ is 
 invariant under the dynamics of \eqref{DNLS1a}
 for each $k \in \Z_0$.

\end{theorem}

Note that we have $\F L_0^{s, p}(\T) \subset  L^2_0(\T): = \F L_0^{0, 2}(\T)$
for $(s, p)$ satisfying the condition in Theorem \ref{THM:main}.
Namely, $L^2_0(\T)$ is the largest space where Theorem \ref{THM:main} holds.
We point out that the quantity 
$\Q_k(t) = \Q_k[u](t) $ constructed in Theorem \ref{THM:main}
is not a conservation law in the usual sense.
Given a global solution $u$ constructed in Theorem \ref{THM:main}, 
we have  $\Q_k[u](t) = \Q_k[u](0)$ for all $t \in \R$ and $k \in \Z_0$.
However, the definition of $\Q_k(t)$ depends
on the information of the solution $u$ on the entire interval $[0, t]$,
whereas a conservation law in the usual sense  depends only 
on the information of the solution $u$ at this specific time $t$.
See \eqref{NF4} for the definition of $\Q_k(t)$.

The smallness condition in Theorem \ref{THM:main} is sharp.
In fact, we construct an explicit finite time blowup solution
for a non-small mean-zero initial condition.

\begin{theorem}[finite time blowup solution]\label{THM:2}
There exists a mean-zero function $\phi \in L_0^2(\T)$
and $t_* >0$
such that 
the corresponding solution $u$ to \eqref{DNLS1a} on the time interval $[0, t_*)$ with $u|_{t = 0} = \phi$
satisfies
\[ \lim_{t \to t_*-} \| u(t) \|_{L^p(\T)} = \infty\]

\noi
for any $1 \leq p \leq \infty$.
In particular, we have
\[ \lim_{t \to t_*-} \| u(t) \|_{\F L_0^{s, p}(\T)} = \infty\]

\noi
for $(s, p)$ satisfying the condition in Theorem \ref{THM:main}.

\end{theorem}

Recall the following Galilean invariance for \eqref{DNLS1a} on $\T$;
if $u$ is a solution to \eqref{DNLS1a} on $\T$, then
\begin{align}\label{Galilei}
u_c(t, x) : =  u(t, x + ct) + c
\end{align}

\noi
is also a solution to \eqref{DNLS1a} for any $c \in \R$.	
This Galilean invariance allows
us to convert a function of a real-valued spatial mean
to a mean-zero function.
Namely, given an initial condition $\phi \in L^2(\T)$
with $\ft \phi(0) \in \R$, we  can convert it into 
a mean-zero initial condition $\phi_c= \phi + c$ with $c = \ft \phi (0)$	
and construct a global solution $u_c \in C(\R; L^2_0(\T))$ to \eqref{DNLS1a} with $u_c|_{t = 0} = \phi_c$
as long as the transformed initial condition $\phi_c$
satisfies the smallness condition stated in Theorem \ref{THM:main}.
Then, by inverting \eqref{Galilei}, 
we obtain a global solution $u\in C(\R; L^2(\T))$ to \eqref{DNLS1a}
with $u |_{t = 0} = \phi$.

Let us now point out how Theorem \ref{THM:main}
does not contradict the ill-posedness result in \cite{Christ} mentioned above.
Christ  proved the norm inflation \eqref{NI}, 
using initial data of the form $\phi_\eps(x) = i A_\eps + i B_\eps e^{iN_\eps x}$
with $A_\eps, B_\eps > 0$ and  $N_\eps \in \mathbb N$.
In particular, the spatial mean of $\phi_\eps$ is purely imaginary
and hence this ill-posedness result is not applicable to our setting.
Lastly, note that the Galilean invariance \eqref{Galilei}
does not allow us to convert a function of a purely imaginary spatial mean
into a function of a real spatial mean
since it would involve non-real $c \in i \R$.

\medskip

There are two main ingredients 
in the proof of Theorem \ref{THM:main}:
(i)  normal form reductions
and (ii) (modified) Cole-Hopf transformation.
The normal form approach will be used
to handle local-in-time analysis in the low regularity setting,
while the modified Cole-Hopf transformation will
be used
to construct smooth local-in-time solutions
as well as obtain a global-in-time a priori estimate.

\smallskip

\noi
 {\bf (i) Normal form reductions.}
The basic idea of normal form reductions
from dynamical systems
is to renormalize 
the flow by removing non-resonant terms
in a nonlinearity at the expense of introducing
higher order nonlinear terms.
This is achieved by introducing a suitable new unknown; see \cite{A, Niko}. 
 We perform normal form reductions
at the level of the interaction representation $v(t) : = e^{it \dx^2} u(t)$.
If $u$ is a smooth solution to \eqref{DNLS1a} with spatial mean 0, 
then the interaction representation $v$ satisfies
\begin{align}
\partial_t \ft v(k) 
& = \frac{ik}{2} \sum_{\substack{k = k_1 + k_2\\k_1, k_2 \ne0}} e^{i (k^2 - k_1^2 - k_2^2) t} \,\ft v({k_1})\ft v({k_2})
 = \frac{ik}{2} \sum_{\substack{k = k_1 + k_2\\k_1, k_2 \ne0}} e^{2 i k_1 k_2 t} \,\ft v({k_1})\ft v({k_2})
\label{intro1}
\end{align}

\noi
for each $k \in \Z_0$.
Here, we used the fact that 
$\Phi(\bk) : =  k^2 - k_1^2 - k_2^2 = 2k_1k_2 $ under $k = k_1 + k_2$.
Then, by performing differentiation by parts, i.e.~integration by parts without an integral sign, 
we obtain 
\begin{align}
\partial_t \ft v(k) 
& = 
\dt \bigg[ \frac{ k}{2 } \sum_{\substack{k = k_1 + k_2\\k_1, k_2 \ne0}} 
e^{ i \Phi(\bk) t}  \frac{\ft v(k_1)\ft v(k_2)}{\Phi(\bk)} \bigg] 
- \frac{ k}{2 } \sum_{\substack{k = k_1 + k_2\\k_1, k_2 \ne0}} 
e^{i \Phi(\bk)  t}  \, \frac{\partial_t \big\{\ft v(k_1)\ft v(k_2)\big\}}{\Phi(\bk)} \notag \\
&=: \partial_t \ft {\N^2(v)}(k) + \ft {\B^3(v)}(k).
\label{intro2}
\end{align}

\noi
for each $k \in \Z_0$.
By this process, we have the modulation function 
$\Phi(\bk) = 2k_1 k_2$ appearing in the denominators,
yielding gain of derivatives.\footnote{Compare this with
the standard Fourier restriction norm method, 
where one gains only $\sim \frac 12$-power of the modulation function $\Phi(\bk)$.}
The price we have to pay is that 
 $w := v - \N^2(v)$ satisfies
$\dt w = \B^3(v)$, where 
$\B^3(v)$ now consists of a cubic nonlinearity
in view of \eqref{intro1} and \eqref{intro2}.
In order to handle
$\B^3(v)$, we need to perform differentiation by parts again.
In fact, 
we will iterate this differentiation by parts process 
infinitely many times (with suitable adjustments at each step)
in Section \ref{SEC:NF}
to renormalize \eqref{DNLS1a} into a simpler equation (see the normal form equation \eqref{NF2} below),
where nonlinear analysis can be carried out
with simple tools such as H\"older's and Young's inequalities.
In particular, we do not employ any auxiliary 
function spaces such as Strichartz spaces and the $X^{s, b}$-spaces, 
allowing us to prove 
 the unconditional uniqueness in Theorem \ref{THM:main}.

In \cite{BIT}, Babin-Ilyin-Titi introduced
this normal form approach via the differentiation by parts
for constructing solutions to 
 the KdV equation on $\T$.
Subsequently, this idea was applied to other dispersive PDEs
 \cite{GKO, GST, Kwon-Oh},
 allowing us to construct solutions to dispersive PDEs
 without relying on the Fourier restriction norm method. 
In \cite{GKO}, 
we further developed this idea and 
successfully implemented  an infinite iteration 
scheme of the so-called Poincar\'e-Dulac normal form reductions for the cubic NLS on $\T$
in the low regularity setting.
This normal form approach has various applications
such as exhibiting
 nonlinear smoothing \cite{ET}
 and 
establishing a good energy estimate \cite{OTz}.

\smallskip

\noi
 {\bf (ii) (modified) Cole-Hopf transformation.}
It is well known that 
the Cole-Hopf transformation \cite{Cole, Hopf}
transforms 
the viscous Burgers equation on $\R$:
\begin{equation} 
\dt u = \dx^2 u + u\dx u
\label{Burgers}
\end{equation}
into the linear heat equation.
More precisely, 
if $u$ is a smooth solution to  the viscous Burgers equation
\eqref{Burgers}, 
 then 
 \[ w(t,x) := e^{-\frac 12 \int_{-\infty}^x u(t,y)  dy }\]
 
 \noi 
solves the linear heat equation on $\R$: $ \dt w = \dx^2 w $. 
A similar trick works for  the quadratic dNLS \eqref{DNLS1a} on $\R$. 
If $u \in C^\infty_{t, \text{loc}}(\R; \S(\R))$ solves 
dNLS \eqref{DNLS1a} on $\R$, 
then 
\begin{align}
 w (t, x) := e^{-\frac i2 \int_{-\infty}^x u(t, y)dy } 
\label{gauge1}
 \end{align}

 \noi
  solves the linear Schr\"odinger equation:
$\dt w = i \dx^2 w$. 
See, for example,  \cite{TAO}.
Similar gauge transformations played an important role
in studying other dispersive PDEs with derivative nonlinearities 
such as the cubic DNLS \eqref{DNLS2} and the Benjamin-Ono equation 
\cite{HO, TaoBO}.

In the periodic case, 
these gauge transformations
have been suitably adjusted to study 
well-posedness 
of   the cubic DNLS \eqref{DNLS2} and the Benjamin-Ono equation.
See \cite{Herr, Molinet}.
For our problem, 
a naive approach would be the following.
Given a mean-zero function $\phi$ on $\T$, 
we define the Cole-Hopf transformation  $\GG_0$ 
by  
\begin{equation*} 
  \GG_0[\phi] : = e^{-\frac{i}{2} \J(\phi)},  
\end{equation*}

\noi
where $\J(\phi)$ is the mean-zero primitive of $\phi$ defined by 
$ \J(\phi)_k := \frac{\phi_k}{ik}$ if $k \ne 0 $, and $\J(\phi)_k := 0$ if  $k=0$. 
Note that $\GG_0[\phi]$ is a periodic function on $\T$, since $\J(\phi)$ is periodic.
Given   a smooth mean-zero solution $u$ to dNLS \eqref{DNLS1a} on $\T$, 
let  $w(t) := \GG_0[u(t)]$.
It turns out that $w$ does not quite satisfy the linear Schr\"odinger equation.
Hence, we need to introduce a suitable adjustment
to  the Cole-Hopf transformation; see \eqref{CH3} below.

Once we appropriately define the modified Cole-Hopf transformation $\GG$, 
we can transform any smooth mean-zero solution $u$ to \eqref{DNLS1a}
to a solution $W = \GG[u]$ to the linear Schr\"odinger equation.
Note that, even if $u$ is known to satisfy \eqref{DNLS1a}
locally in time, the gauged function $W$ exists globally in time
as a solution to the linear equation.
Then, an important  question is to investigate
the invertibility of this transformation.
It turns out that 
the inverse transformation is given by 
\begin{equation*} 
u(t, x) = \GG^{-1}[W](t, x) := 2i \frac{\dx W(t, x)}{W(t, x)}
\end{equation*}

\noi
for a smooth mean-zero solution $u$ to \eqref{DNLS1a}.
Thus, $W(t, x) \ne 0$ guarantees the invertibility of the modified Cole-Hopf transformation.
In Subsection \ref{SUBSEC:GWP1}, 
we discuss possible sufficient conditions for the invertibility in details 
and construct smooth global solutions to \eqref{DNLS1a}.

Interestingly, by interpreting this modified Cole-Hopf transformation
from a geometric point of view, 
we obtain a necessary condition for possible spatial means of initial data $\phi$ for \eqref{DNLS1a}.
Given a smooth solution $u$ to \eqref{DNLS1a} on some time interval $I$, 
$u(t) $ is a periodic function for each $t \in I$.
In particular, it is a closed loop in the complex plane.
Then, the transformed function $W(t) = \GG[u](t)$
is also a closed loop in the complex plane.
Therefore, $W$ must have a well defined index 
around the origin.
This observation leads to $\int_\T \phi(x) dx = 4\pi n$, $n \in \Z$.
See 
Remark \ref{REM:winding}.

\smallskip

We conclude this introduction by mentioning a hidden connection 
between the normal form reductions and the modified Cole-Hopf transformation.
In Section \ref{SEC:NF}, 
we obtain an infinite series
as a result of an infinite iteration of normal form reductions. 
It turns out that this series is nothing but the Taylor expansion of (the derivative of
the interaction representation of) the transformed function  $W = \GG[u]$
solving the linear Schr\"odinger equation.
Namely, the two seemingly different approaches
reduce \eqref{DNLS1a} to the same linear Schr\"odinger equation (up to a differentiation); see Remark \ref{REM:linear}.
Such a 
reducibility to a linear equation
by normal form reductions is an important question, closely related to integrability.
For example, see
Nikolenko \cite{Niko}
in the smooth setting.

\smallskip

This paper is organized as follows.
In Section \ref{SEC:NF}, we implement an infinite iteration scheme 
of normal form reductions and rewrite \eqref{DNLS1a}
as  the normal form equation \eqref{NF2}.
By establishing nonlinear estimates, we  prove unconditional local well-posedness of the normal form equation. 
In Section \ref{SEC:CH}, we introduce a modified Cole-Hopf transformation,
converting \eqref{DNLS1a} into the linear Schr\"odinger equation. 
By investigating a sufficient  condition
for inverting the modified Cole-Hopf transform, 
we prove small data global well-posedness of \eqref{DNLS1a}
in $L^2_0(\T)$.
In Section \ref{SEC:GWP}, we put together the results
from Sections \ref{SEC:NF} and \ref{SEC:CH} and prove Theorem \ref{THM:main}.
 In  Appendix \ref{SEC:APP}, we present the proof of Theorem \ref{THM:2}
 by constructing an explicit example.

In view of the time-reversibility of the equation \eqref{DNLS1a}, 
we only consider positive times in the following.

\section{Normal form reductions}
\label{SEC:NF}

In this section, we perform  normal form reductions in an iterative manner
and rewrite \eqref{DNLS1a} as an equation involving infinite series.
Then, we establish an a priori estimate on solutions 
in the Fourier-Lebesgue spaces.
As mentioned in Section \ref{SEC:intro}, 
our main approach is to apply differentiation by parts iteratively
and reduce an equation into a {\it simpler} form at each step.
This procedure can be seen as an infinite dimensional version of the Poincar\'e-Dulac 
normal form reduction\footnote{In fact, our approach 
for \eqref{DNLS1a} is 
an infinite dimensional analogue of the Poincar\'e normal form reduction.  
Namely, there is no resonant term (modulo the correction terms) at each step of the iteration.}
for a finite dimensional system of ODEs \cite{A}. 
On the one hand, 
two iterations were sufficient
in \cite{BIT, GST, Kwon-Oh}.
On the other hand, 
we needed to iterate normal form reductions infinitely many times
in \cite{GKO} to prove unconditional global well-posedness for the cubic NLS on $\T$.
In the following, 
we also set up an infinite iteration scheme for \eqref{DNLS1a}.
As we see below, 
our argument for \eqref{DNLS1a} is somewhat simpler than that in \cite{GKO}.
This is due to miraculous symmetrizations at each step of the normal form reductions, 
which allows us to write an equation at each step in a very simple form.
In particular, there will be no resonant term  in the equation at each step
(modulo the correction terms $\RR^n(v)$
appearing in \eqref{I3z},  \eqref{I4z}, and \eqref{rec3a}).
Note that all the analysis in this section is of local-in-time nature.

\subsection{Formal iteration argument}
In the following, we first perform normal form reductions
at a  formal level.
Namely, in this subsection, we do not worry about issues such as the convergence
of infinite series, switching the time derivative with an infinite series, etc.
These issues will be addressed in the next subsection.

Before proceeding further, let us introduce some notations.
Given a periodic function $f \in L^2(\T)$, 
we define its Fourier coefficient $f_k$ by 
\[ f_k := \frac{1}{2\pi} \int_{\T}f(x) e^{-ikx}dx. \]

\noi
The Fourier inversion formula states
\[ f(x) = \sum_{k \in \Z} f_k \,e^{ikx}.\]

\noi
Then, we can rewrite \eqref{DNLS1a}  as 
\begin{align} 
\partial_t u_k = -i k^2 u_k + \frac{ik}{2}  \sum_{k = k_1+k_2} u_{k_1}u_{k_2}. 
\label{DNLS3}
\end{align}

\noi
In the following, we assume that $\int \phi \, dx = 0$.
Hence, it follows from the conservation of mean \eqref{mean}
that $u_0(t) = 0$ for all $t\in \R$
and it is understood that the summation is over non-zero frequencies
$\Z_0 := \Z \setminus \{ 0\}$.

We now introduce  the interaction representation $v$ by setting
\begin{align}
v(t) := S(-t) u(t), 
\label{IR1}
\end{align}

\noi
where $S(t) = e^{it \dx^2}$.
In terms of the Fourier coefficients, we have 
$v_k(t) = e^{i k^2 t} u_k(t)$. 
Then, the equation \eqref{DNLS3} becomes
\begin{equation}
\partial_t v_k = \frac{ik}{2} \sum_{k = k_1 + k_2} e^{i (k^2 - k_1^2 - k_2^2) t} \,v_{k_1}v_{k_2}. 
\label{DNLS4}
\end{equation}

Given 
$\bk = (k_1, k_2, \cdots) \in \Z_0^\infty$
and $n \in \mathbb{N}$, 
we define a length $|\bk|_n$ and a modulation function $\Phi_n(\bk)$
by setting 
\begin{align}
|\bk|_n &:= k_1 + k_2 + \cdots + k_n, \notag\\
\Phi_n(\bk) &:= \bigg( \sum_{i=1}^n k_i  \bigg)^2 - \sum_{i=1}^n k_i^2 = \sum_{1 \leq i < j \leq n} 2 k_i k_j.
\label{Phi}
\end{align}

\noi
By definition, we set $\Phi_1(\bk) = 0$.
Note that the modulation function $\Phi_2(\bk)$
appears as a phase factor in \eqref{DNLS4}.
As we see below, 
the modulation function  $\Phi_{n+1}(\bk)$ appears as a phase factor
at the $n$th  step of the iteration argument.
On the one hand, 
if $ \Phi_n(\bk)\ne 0$, corresponding to the so-called {\it non-resonant} term, and it is large in particular,
 then we expect some cancellation
 under a time integration.
 On the other hand, 
if $\Phi_n(\bk)=0$, corresponding to the so-called {\it resonant} term, 
then there is no cancellation under a time integration.
In the following, we exploit symmetrizations at each step
and show that all the terms we obtain are indeed non-resonant.

With $\Phi_2(\bk)  = 2k_1k_2$, we can rewrite \eqref{DNLS4} as 
\begin{align}
\label{I2}
 \partial_t v_k = \frac{ k}{4} 
 \sum_{|\bk|_2 = k} i \Phi_2(\bk) 
e^{i \Phi_2(\bk)t} \frac{v_{k_1} v_{k_2}}{k_1 k_2} =: \I^2(v)(k). 
\end{align}

\noi
Note that there is no resonant term $\Phi_2(\bk)=0$ in \eqref{I2} because $\bk \in \Z_0^\infty$.
For conciseness, 
 we use $\I^2_k(v)$ to denote the $k$th Fourier coefficient $\I^2(v)(k)$ of 
the multilinear form $\I^2(v)$
in the following.
A similar comment applies to other multilinear expressions.

We apply  differentiation by parts 
 and obtain
\begin{align}
\I_k^2(v) &= \dt 
\bigg[ \frac{ k}{4 } \sum_{|\bk|_2=k} e^{i \Phi_2(\bk) t}  \frac{v_{k_1}v_{k_2}}{k_1 k_2} \bigg] 
- \frac{ k}{4 } \sum_{|\bk|_2 = k} e^{i \Phi_2(\bk) t}  \, \frac{\partial_t (v_{k_1}v_{k_2})}{k_1 k_2} \notag \\
&=: \partial_t \N_k^2(v) + \B_k^3(v).
\label{I2a}
\end{align}

\noi
By symmetry, 
we assume that the time derivative in $\partial_t (v_{k_1}v_{k_2})$
falls only on $v_{k_2}$
and we double the contribution.
From \eqref{I2a} with \eqref{DNLS4}, we have 
\begin{align}
\B_k^3(v) 
& = -2\cdot \frac k 4 \sum_{|\bk|_2 = k} e^{i \Phi_2(\bk) t} \frac{v_{k_1} \partial_t v_{k_2}}{k_1 k_2}  \notag\\
& = -  \frac{ ik}{4} \sum_{k = k_1+k_2} e^{i \Phi_2(\bk) t} \, \frac{v_{k_1}}{k_1} 
\sum_{\substack{ k_2 = m_1 + m_2\\k_2 \ne 0}}  
e^{i (k_2^2 - m_1^2 - m_2^2) t} \,v_{m_1} v_{m_2}.
\label{I3}
\end{align}

\noi
In order to apply a symmetrization, 
we need to add and subtract the contribution from $m_1 + m_2 = 0$.
For this purpose, define $\RR_k^2(v)$ and $\I_k^3(v)$ by 
\begin{align}
\RR_k^2(v)  = \frac{i}{4} v_kM(u)
=  \frac{i}{4} v_k  \sum_{m \in \Z_0} e^{-2i m^2 t} v_m v_{-m}
\quad \text{and}\quad 
\I_k^3(v) :\! &=   \B_k^3(v)-\RR_k^2(v), 
\label{I3z}
\end{align}

\noi
where $M(u)$ is   given by 
\begin{align}
M(u): = \P_0 [u^2] = \frac{1}{2\pi} \int_\T u^2 dx.
\label{mass}
\end{align}

\noi
Here,  $\P_0$ denotes the Dirichlet projection onto the zeroth frequency.
We point out  that $M(u) \ne \|u\|_{L^2}^2$.
From \eqref{I2a} and \eqref{I3z}, we have 
\begin{align}
 \I_k^2(v) =\partial_t \N_k^2(v) + \RR_k^2(v) +  \I_k^3(v).
\label{I3x}
\end{align}

\noi
By symmetrization in $\{k_1, k_2, k_3\}$, we have
\begin{align}
\I^3_k(v) : = \B_k^3(v) - \RR_k^2(v)  
&= - \frac{ik}{4} \sum_{k=k_1+m_1+m_2} 
e^{i (k^2-k_1^2-m_1^2-m_2^2) t} \,\frac{v_{k_1}}{k_1} v_{m_1} v_{m_2}\notag  \\
&= - \frac{ ik}{8} \sum_{|\bk|_3 = k} 2k_2 k_3 \,e^{i \Phi_3(\bk) t} \,\frac{v_{k_1}v_{k_2}v_{k_3}}{k_1 k_2 k_3} \notag \\
&= - \frac{ k}{4 \cdot {3 \choose 2}} \sum_{|\bk|_3 = k}i  \Phi_3(\bk) e^{i \Phi_3(\bk) t} \,\frac{v_{k_1} v_{k_2} v_{k_3}}{k_1 k_2 k_3} \notag\\
&= - \frac{ k}{2^2 \cdot 3! } \sum_{|\bk|_3 = k}i  \Phi_3(\bk) e^{i \Phi_3(\bk) t} \,\frac{v_{k_1} v_{k_2} v_{k_3}}{k_1 k_2 k_3}.  
\label{I3y}
\end{align}

\noi
Note that  the symmetrization cancelled out
the resonant terms ($ \Phi_3(\bk) = 0 $) 
and there is no resonant term in the final expression 
in \eqref{I3y}. 
Therefore, we can perform differentiation by parts on $\I^3_k(v)$

\begin{remark} \label{REM:tdepend} \rm

Due to the presence of 
$e^{i \Phi_n(\bk) t}$ 
in their definitions, 
the multilinear expressions $\I^n(v)$ and $\N^n(v)$ 
are non-autonomous and in fact depend on $t$. 
Thus, strictly speaking, 
we should denote them by  $\I_k^n(t)(v(t))$ and $\N_k^n(t)(v(t))$. 
For simplicity of notations, however, we suppress such $t$-dependence
when there is no confusion.
The same comment applies to $\B^n(v)$ and $\RR^n(v)$.
\end{remark}

Differentiating $\I^3_k(v)$ by parts  yields 
\begin{align}
\I_k^3(v) &= \partial_t 
\bigg[ \frac{-k }{2^2 \cdot 3! } \sum_{|\bk|_3 = k} e^{i \Phi_3(\bk) t} \frac{v_{k_1} v_{k_2} v_{k_3}}{k_1 k_2 k_3} \bigg] 
+ \frac{k }{2^2 \cdot 3! }
\sum_{|\bk|_3 = k} e^{i \Phi_3(\bk) t} \frac{\partial_t ( v_{k_1} v_{k_2} v_{k_3} )}{k_1 k_2 k_3} \notag \\
&=: \partial_t \N_k^3(v) + \B_k^4(v).
\label{I3a}
\end{align}

\noi
Then, 
by symmetry as above, we have 
\begin{align}
\B_k^4(v) 
&= 
3 \cdot  \frac{k }{2^2 \cdot 3! }
\sum_{|\bk|_3 = k} e^{i \Phi_3(\bk) t} \frac{ v_{k_1} v_{k_2} \dt v_{k_3} }{k_1 k_2 k_3} \notag \\
& = 
 \frac{i k }{2^4 }
\sum_{|\bk|_3 = k} e^{i \Phi_3(\bk) t} \frac{ v_{k_1} v_{k_2}  }{k_1 k_2} 
\sum_{\substack{ k_3 = m_1 + m_2\\k_3 \ne 0}}  
e^{i (k_3^2 - m_1^2 - m_2^2) t} \,v_{m_1} v_{m_2}.
\label{I4y}
\end{align}

\noi
Define $\RR_k^3(v)$ and $\I_k^4(v)$ by 
\begin{align}
\RR_k^3(v)  = \frac{- ik }{2^4} M(u) \sum_{|\bk|_2 = k} e^{i \Phi_2(\bk)t} \frac{v_{k_1}v_{k_2}}{k_1k_2}
\quad \text{and} \quad
\I_k^4(v) \! &=   \B_k^4(v)-\RR_k^3(v).
\label{I4z}
\end{align}

\noi
From \eqref{I3a} and \eqref{I4z}, we have 
\begin{align}
 \I_k^3(v) =\partial_t \N_k^3(v) + \RR_k^3(v) +  \I_k^4(v).
\label{I4x}
\end{align}

\noi
From \eqref{I4y}, \eqref{I4z}, 
and symmetrization, we have
\begin{align*}
\I_k^4(v) 
&  = \frac{i k }{2^5} \sum_{|\bk|_4 = k} 2 k_3 k_4 e^{i \Phi_4 (\bk) t} \, \frac{v_{k_1} v_{k_2} v_{k_3} v_{k_4}}{k_1 k_2 k_3 k_4}
\notag \\
& = \frac{k }{2^5 \cdot {4 \choose 2}
} \sum_{|\bk|_4 = k} i \Phi_4 (\bk) e^{i \Phi_4 (\bk) t} \, \frac{v_{k_1} v_{k_2} v_{k_3} v_{k_4}}{k_1 k_2 k_3 k_4}
\notag \\
&= \frac{k}{2^3 \cdot 4!} 
\sum_{|\bk|_4 = k} i \Phi_4 (\bk) e^{i \Phi_4 (\bk) t} \, \frac{v_{k_1} v_{k_2} v_{k_3} v_{k_4}}{k_1 k_2 k_3 k_4}.
\label{I4}
\end{align*}

Figure \ref{FIG:iteration} shows  this iteration procedure schematically. 
We iterate this process of differentiation by parts and symmetrization indefinitely.

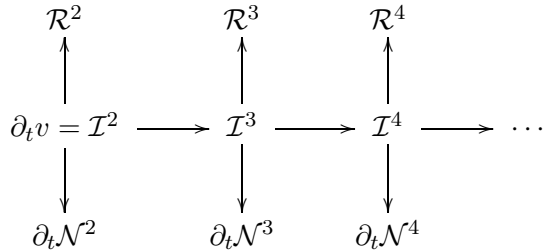
\begin{figure}[h]

\begin{equation*} 
\xymatrix{
\RR^2   \ar@{<-}[d] 
&
\RR^3   \ar@{<-}[d] 
 & \RR^4   \ar@{<-}[d] 
&   \\
\ \dt v = \I^2 \  \ar@{->}[d]   \ar@{->}[r] 
&
\ \I^3  \ \ar@{->}[d] \ar@{->}[r] 
 &\  \I^4  \  \ar@{->}[d] \ar@{->}[r] 
&  \ \cdots \\
 \dt \N^2
&
 \dt \N^3  
 &\dt \N^4  & 
}\end{equation*}

\caption{Schematic diagram of the iteration process.} \label{FIG:iteration}

\end{figure}

\noi
The following lemma describes the equation we obtain after the $n$th step.

\begin{lemma}  \label{LEM:iteration}
Let $n \in \mathbb N$.  Then, the equation  \eqref{DNLS4} can be rewritten 
as follows:
\begin{align}\label{finiteNF}
\partial_t v_k &= \I_k^2(v) =\partial_t \N_k^2(v) + \RR_k^2(v) +  \I_k^3(v) \notag \\
&= \cdots = \partial_t \bigg(\sum_{j = 2}^n \N_k^j(v) \bigg) 
+ \sum_{j = 2}^n \RR_k^j(v) 
+ \I_k^{n+1}(v), 
\end{align} 

\noi
where  $\I^n(v) $, $ \N^n(v)$,  
and $\RR^n(v)$
satisfy the following recursive relation: 
\begin{equation} \label{rec1}
\I_k^n(v) = \partial_t \N_k^n(v) + \RR_k^n(v) +\I_k^{n+1}(v), \quad n \geq 2.
\end{equation}

\noi
Moreover, the Fourier coefficients of  
$\I^n(v) $, $ \N^n(v)$, 
and $\RR^n(v)$,  $n\geq 2$,  are given by 
\begin{align}
\I_k^n(v) &= \frac{(- 1)^n k}{2^{n-1}\cdot n! } \sum_{|\bk|_n = k} i \Phi_n(\bk)  e^{i \Phi_n(\bk) t} 
\prod_{j=1}^n \frac{v_{k_j}}{k_j}, \label{rec2}\\
\N_k^n(v) &= \frac{(- 1)^n k}{2^{n-1}\cdot n! } \sum_{|\bk|_n = k} e^{i \Phi_n(\bk) t}
\prod_{j=1}^n \frac{v_{k_j}}{k_j}, \label{rec3}\\
\RR_k^n(v) &= \frac{-i}{4} M(u) \cdot  \N_k^{n-1}(v) .
\label{rec3a}
\end{align}

\noi
with the convention that 
\begin{align}
\N^1(v) = -v.
\label{rec3b}
\end{align}
\end{lemma}

\begin{remark}\rm
We point out that the factor $ n!$ appears in the 
denominator of  $\N^n(v)$. 
This corresponds to a functional version of the Taylor expansion of an exponential function. See Section~\ref{SEC:CH}.
\end{remark}

\begin{proof}
We argue by induction.
The case $n = 1, 2, 3$ follows from \eqref{I2}, \eqref{I3x},  and \eqref{I4x}.
Now, suppose that \eqref{finiteNF} -- \eqref{rec3a} hold for some $n \in \mathbb N$.
Differentiating $\I_k^{n+1}(v) $ by parts, we have
\begin{align}
\I_k^{n+1}(v) 
&= 
\dt \bigg[\frac{(- 1)^{n+1} k}{2^{n}\cdot (n+1)! } \sum_{|\bk|_{n+1} = k} e^{i \Phi_{n+1}(\bk) t}
\prod_{j=1}^{n+1} \frac{v_{k_j}}{k_j}\bigg]  \notag\\
& \hphantom{X}+ 
\frac{(- 1)^{n+2} k}{2^{n}\cdot (n+1)! } \sum_{|\bk|_{n+1} = k}  e^{i \Phi_{n+1}(\bk) t} 
\dt \bigg( \prod_{j=1}^{n+1} \frac{v_{k_j}}{k_j}\bigg)
=: \partial_t \N_k^{n+1}(v) + \B_k^{n+2}(v).
 \label{rec5}
\end{align}

\noi
The first term $\N_k^{n+1}(v)$ readily satisfies \eqref{rec3}.
Let $\RR^{n+1}_k(v)$ be as in \eqref{rec3a}.
By symmetrization as before, 
 we have
\begin{align*}
\mathcal{I}_k^{n+2}(v)
:& = \mathcal{B}_k^{n+2}(v) - \RR^{n+1}_k(v)\\
&= (n+1)\cdot 
\frac{(- 1)^{n+2} i k}{2^{n+1}\cdot (n+1)! } \sum_{|\bk|_{n+1} = k}  e^{i \Phi_{n+1}(\bk) t} 
\prod_{j=1}^{n} \frac{v_{k_j}}{k_j} 
\notag\\
&  \hphantom{XXXXXX}
\times 
\sum_{k_{n+1} = m_1 + m_2}  
e^{i (k_{n+1}^2 - m_1^2 - m_2^2) t} \,v_{m_1} v_{m_2}\\
& 
= 
 \frac{(- 1)^{n+2} k}{2^{n}\cdot n! \cdot 4 \cdot {n +2\choose 2}} 
\sum_{|\bk|_{n+2} = k} i  \Phi_{n+2}(\bk)e^{i \Phi_{n+2}(\bk) t} 
 \prod_{j=1}^{n+2} \frac{v_{k_j}}{k_j}, 
\end{align*}

\noi
which agrees with \eqref{rec2}.
Lastly, \eqref{finiteNF} for the $(n+1)$-st step follows from 
\eqref{finiteNF} for the $n$th step
 and \eqref{rec5}.
\end{proof}

The proof of Lemma \ref{LEM:iteration} shows that, at each step
of the normal form reductions,  
the symmetrization completely eliminates the resonant terms
in $\I^{n}(v)$. 
Now, suppose that $\I_k^{n+1}(v) \to 0 $ as $ n \to \infty$.
Then, by taking $n \to \infty$ in \eqref{finiteNF}, 
we {\it formally} arrive at the limit equation:
\begin{equation} 
\partial_t v_k = \partial_t \bigg( \sum_{n=2}^\infty \N_k^n(v) \bigg)
+ \sum_{n=2}^\infty \RR_k^n(v).
\label{NF1}
\end{equation}

\noi
Integrating \eqref{NF1} in time and applying the Fourier inversion formula, we 
obtain the following  {\it normal form equation}:
\begin{equation} 
v(t) = \phi + \sum_{n=2}^\infty  \N^n(v(t)) - \sum_{n=2}^\infty\N^n(\phi) 
+ \int_0^t  \sum_{n=2}^\infty \RR^n(v(t')) dt'.
\label{NF2}
\end{equation}

\noi
In the following, we refer to \eqref{NF2} as the {\it normal form equation}.

\begin{remark}\label{REM:cons}\rm

(i)	
With \eqref{rec3a} and \eqref{rec3b}, we can rewrite
\eqref{NF1} as 

\begin{equation} 
 \partial_t \bigg( \sum_{n = 1}^\infty \N_k^n(v) \bigg)
= \frac{i}{4} M(u) 
\bigg( \sum_{n = 1}^\infty \N_k^n(v) \bigg), 
\label{NF3}
\end{equation}

\noi
where $u(t) = e^{it \dx^2} v(t) $.
Given a smooth function $v$ on $\R\times \T$, 
define 
\begin{align}
\Q_k(t) = \Q_k([0, t])
:= e^{- \frac{i}{4}\int_0^t M(u)(t') dt'}
\sum_{n = 1}^\infty \N_k^n(v)
\label{NF4}
\end{align}

\noi
for $k \in \Z_0$.
Then, 
given a global solution $v$ to \eqref{NF2}, 
the equation \eqref{NF3} formally yields
\begin{align}
\Q_k(t) = \Q_k(0)
\label{NF5}
\end{align}

\noi
 for any $t \in \R$.
 Hence, we found an infinite sequence $\{\Q_k(t)\}_{k \in \Z_0}$
 of the quantities that are formally invariant under the dynamics of
 the normal form equation \eqref{NF2}.
 We point out, however,  that $\Q_k$
is {\it not} a conservation law in the usual sense.
Namely, 
the definition \eqref{NF4} of $\Q_k(t)$
 depends on the information of
 the solution $v$ on 
 the entire interval  $[0, t]$.

\smallskip

\noi

(ii)
Define $q(t, x)$  by setting 
\begin{align*}
q_k(t) =  e^{-ik^2t } \Q_k(t), \quad k \in \Z_0.
\end{align*}

\noi
Then, it follows from \eqref{NF5} that 
$q_k(t) =    e^{-ik^2t } \Q_k(0)$.
Namely, 
$q$ is a solution to the linear Schr\"odinger equation:
$\dt q  = i \dx^2 q$.
Therefore, 
via the infinite iteration of normal form reductions, 
we have found, at least at a formal level, 
 a transformation $u \mapsto q$, 
 mapping a solution to \eqref{DNLS1a}
 to a solution to the linear Schr\"odinger equation. 
See Section \ref{SEC:CH} for more discussion on this issue.

\end{remark}

%

\subsection{Analysis of the normal form equation}

In the previous subsection, 
we reduced the original dNLS \eqref{DNLS1a}
to the normal form equation \eqref{NF2}
through an infinite sequence of normal form reductions.
In this subsection, we first justify
the formal computations performed in the previous section, 
at least for smooth solutions to \eqref{DNLS1a}.
Then, we prove local well-posedness
of \eqref{NF2}
by establishing certain multilinear estimates.
In Section \ref{SEC:CH}, 
we transfer this well-posedness of \eqref{NF2}
to the original problem \eqref{DNLS1a}.

We first introduce a constant $Z_{s,p}$ which we will frequently use.
 
\begin{definition}\label{DEF:Zsp}\rm
For $s,p\in \R$ satisfying $s > -\frac{1}{p}$ and $p \geq 1$, we define a constant
$Z_{s, p}$ by 
\[ Z_{s,p} := \big\| |k|^{-(s+1)} \big\|_{\ell^{p'}(\Z_0)} = \Big[ 2 \,\zeta\big( (s+1) \,p' \big) \Big]^{\frac{1}{p'}}
< \infty, \]

\noi
where $\zeta(\tau) = \sum_{k = 1}^\infty k^{-\tau}$ is the Riemann zeta function 
and $p'$ is the H\"older conjugate of $p$: $\frac 1p + \frac 1{p'} = 1$.
 If $p = 1$, we can also define $Z_{s,p}$ for all $s \geq -1$ by 
\[ Z_{s,1} := \big\| |k|^{-(s+1)} \big\|_{\ell^{\infty}(\Z_0)} = 1 \quad \text{for all $s \geq -1$.} \]

For $s,p\in \R$ satisfying $s > \frac 12 -\frac{1}{p}$ and $p > 2$,
 we also define a constant
$z_{s, p}$ by 
\[ z_{s, p} := Z_{s-1,\frac{2p}{p+2}} < \infty\]

\noi
such that we have
\begin{equation}
 \|f \|_{L^2} \leq z_{s, p} \| f \|_{\F L_0^{s, p}}
\label{embed}
 \end{equation}

\noi
for any mean-zero function $f$ on $\T$.

\end{definition}

The following proposition shows that 
a smooth solution to \eqref{DNLS1a} indeed satisfies the normal form equation 
\eqref{NF2}.

\begin{proposition} \label{PROP:smooth}
Given $T > 0$, 
let $u$ be a smooth  solution to dNLS \eqref{DNLS1a} on $[0, T]$
with a smooth initial data $u|_{t = 0}= \phi$ satisfying $\int_\T \phi \, dx= 0$. 
Then,  its interaction representation $v(t) =  e^{-i t\dx^2} u(t)$ satisfies the normal form 
equation \eqref{NF2} on $[0, T]$.
\end{proposition}

\begin{proof}
We break the proof into two parts.
We first prove that $v$ satisfies \eqref{NF1}
for each $k \in \Z_0$, assuming that $v \in C([0, T]; L^2(\T))$.
Then, by imposing a higher regularity, we 
show that $v$ satisfies the normal form equation \eqref{NF2}.

\smallskip
\noi
{\bf Part 1:}
In view of the time reversibility of \eqref{DNLS1a} and \eqref{NF2}, 
we only consider positive times.
Fix $T > 0$ and define constants $C_0 = C_0(T)$ by
\[
C_0 = C_0 (T):= \sup_{t \in [0, T]} \| u(t) \|_{L^2_x}.
 \]
 
 \noi
Then, from $|v_k(t)| = |u_k(t)|$ and Cauchy-Schwarz inequality,  we have
\begin{align*}
\sup_{t \in [0, T]}\|v_k(t)\|_{\ell^\infty}    \leq C_0
\quad \text{and} \quad 
\sup_{t \in [0, T]}\|k^{-1} v_k(t)\|_{\ell^1} 
 \leq  Z_{0,2} C_0.
\end{align*}

\noi
Hence for given $k\in \Z_0$, 
it follows from \eqref{rec2} with \eqref{Phi} and Young's inequality that
\begin{align}
|\I_k^n(v(t))| &\leq \frac{|k|}{2^{n-1} n!} 
\sum_{|\bk|_n = k} \big| \Phi_n(\bk) \big| \prod_{\l=1}^n \frac{|v_{k_\l}(t)|}{|k_\l|} \notag \\
& \leq \frac{|k|}{2^{n-2} n!} \sum_{1 \leq i <j  \leq n} 
\sum_{|\bk|_n = k} |v_{k_i}(t)| \,|v_{k_j}(t)|  \prod_{ \substack{1 \leq \l \leq n \\ \l \not= i,j}} \frac{|v_{k_\l}(t)|}{|k_\l|}\notag \\
&\leq \frac{|k|}{2^{n-2} n!} 
\sum_{1 \leq i < j \leq n} \|v_k(t) \|_{\ell^2}^2   \| k^{-1}v_k(t) \|_{\ell^1}^{n-2} \notag \\
&\leq |k|    \frac{Z_{0,2}^{n-2} C_0^{n}}{2^{n-1}(n-2)!}
\label{smooth1}
\end{align}

\noi
uniformly in $t \in [0, T]$.
Therefore, we have 
\begin{align}
 \lim_{n \to \infty} \sup_{t \in [0, T]} \I_k^n(v(t))  = 0 
\label{smooth1a}
 \end{align}

\noi
for each $k \in \Z_0$.
By a similar computation with \eqref{rec3}
and the triangle inequality 
$|k| \le \sum_{j = 1}^n |k_j|$, we have
\begin{align}
\sup_{t \in [0, T]}|\N_k^n(v(t))| 
\leq 
\sup_{t \in [0, T]}
\frac{1}{2^{n-1} (n-1)!}  \|v_k(t)\|_{\ell^\infty}  \| k^{-1} v_k(t) \|_{\ell^1}^{n-1} 
\leq  \frac{Z_{0,2}^{n-1} C_0^{n}}{2^{n-1}(n-1)!}. 
\label{smooth2}
\end{align}

\noi
In particular, $\sum_{n = 2}^\infty \N_k^n(v(t))$ converges (absolutely and uniformly in $t \in [0, T]$).

From \eqref{rec3a} and \eqref{smooth2}, we obtain
\begin{align}
\sup_{t \in [0, T]}|\RR_k^n(v(t))| 
\leq  \frac{Z_{0,2}^{n-2} C_0^{n-1}}{2^{n}(n-2)!}
\sup_{t \in [0, T]} M(u(t))
\leq \frac{Z_{0,2}^{n-2} C_0^{n+1}}{2^{n}(n-2)!}
\label{smooth2a}
\end{align}

\noi
for $n \geq 3$.
When $n = 2$, we have 
\begin{align*}
\sup_{t \in [0, T]}|\RR_k^2(v(t))| 
\leq  \frac{1}{4} C_0 \sup_{t \in [0, T]} |v_k(t)|
\leq  \frac{1}{4} C_0^3, 
\end{align*}

\noi
which is precisely \eqref{smooth2a} with $n = 2$.
Hence, 
$\sum_{n = 2}^\infty \RR_k^n(v(t))$ converges absolutely and uniformly in $t \in [0, T]$.

Lastly, it follows from \eqref{rec1}, \eqref{smooth1}, and \eqref{smooth2a}
that $\sum_{n=2}^\infty \partial_t \N_k^n(v(t))$ converges absolutely 
and uniformly in $t \in [0, T]$.
Indeed, we have
\begin{align*}
\sup_{t \in [0, T]} | \partial_t \N_k^n(v(t)) | 
& \leq \sup_{t \in [0, T]}| \I_k^n(v(t)) | + \sup_{t \in [0, T]}| \I_k^{n+1}(v(t)) | 
+ \sup_{t \in [0, T]}|\RR_k^n(v(t))| \\
&\leq 
|k|\bigg[\frac{Z_{0,2}^{n-2} C_0^{n}}{2^{n-1}(n-2)!}
+ \frac{Z_{0,2}^{n-1} C_0^{n+1}}{2^{n}(n-1)!}\bigg]
+ \frac{Z_{0,2}^{n-2} C_0^{n+1}}{2^{n}(n-2)!}.
\end{align*}

\noi
Hence, we have 
$\dt\big(\sum_{n = 2}^\infty \N_k^n(v(t)) )= \sum_{n = 2}^\infty \dt \N_k^n(v(t))$.
Therefore, with \eqref{finiteNF},  we have
\begin{align*}
\bigg|\partial_t v_k(t) & - \dt \bigg(  \sum_{n=2}^\infty   \N_k^n(v(t))\bigg)
 - \sum_{n = 2}^\infty \RR_k^j(v) \bigg| 
  = \bigg|\partial_t v_k(t) - \sum_{n=2}^\infty \partial_t \N_k^n(v(t))
-\sum_{n = 2}^\infty \RR_k^j(v) \bigg| \notag\\
& \leq \lim_{N \to \infty}  \sum_{n=N+1}^\infty 
\big(|\partial_t \N_k^n(v(t))| + |\RR_k^n(v)| \big)
+\lim_{N \to \infty} |\I_k^{N+1}(v(t))|
=0
\end{align*}

\noi
 for each $k\in \Z_0$ and $t \in [0, T]$.  This proves \eqref{NF1} on $[0, T]$.

\medskip

\noi
{\bf Part 2:}
In order to obtain the normal form equation \eqref{NF2}, we need to invert
the Fourier transform and thus need to use a control on a higher regularity.
By a  computation similar to \eqref{smooth2}, we have
\begin{align*}
\sup_{k \in \Z_0} \sup_{t \in [0, T]} |k|^2 |\N_k^n(v(t))| 
\le 
 C_2 \frac{Z_{0,2}^{n-1} C_0^{n-1}}{2^{n-1}(n-1)!},
\end{align*}

\noi
where $C_2 =  C_2 (T):= \sup_{t \in [0, T]} \| u(t) \|_{H^2_x}.$
In particular, $\sum_{k \in \Z_0} \sum_{n = 2}^\infty \N_k^n(v(t))$
converges absolutely 
(and uniformly in $[0, T]$).
Hence, we have 
\begin{align}
 \sum_{n=2}^\infty   \N^n(v(t))-  \sum_{n=2}^\infty   \N^n(\phi)
& =  \sum_{n=2}^\infty\sum_{k \in \Z_0}  \big( \N_k^n(v(t))
-    \N_k^n(\phi)\big) e^{ikx} \notag \\
& = \sum_{k \in \Z_0}
\bigg(  \sum_{n=2}^\infty   \N_k^n(v(t))-  \sum_{n=2}^\infty   \N_k^n(\phi)\bigg)e^{ikx}.
\label{smooth3}
\end{align}

Similarly, 
 $\sum_{k \in \Z_0} \sum_{n = 2}^\infty \RR_k^n(v(t))$
converges absolutely 
(and uniformly in $[0, T]$)
with a uniform bound:
\begin{align*}
\sup_{k \in \Z_0} \sup_{t \in [0, T]} 
\sum_{k \in \Z_0} \sum_{n = 2}^\infty |\RR_k^n(v(t))|
\les \sum_{k \in \Z_0} \frac{1}{|k|^2}
\sum_{n = 2}^\infty
 C_2  \frac{Z_{0,2}^{n-2} C_0^{n}}{2^{n}(n-2)!}
 <\infty.
\end{align*}

\noi
Hence, by Dominated Convergence Theorem, 
we have 
\begin{align}
\int_0^t  \sum_{n=2}^\infty   \RR^n(v(t')) dt'
& =  \sum_{n=2}^\infty \int_0^t\sum_{k \in \Z_0}   \RR_k^n(v(t')) e^{ikx} dt' \notag\\
&  =  \sum_{k \in \Z_0}  \int_0^t \sum_{n=2}^\infty  \RR_k^n(v(t')) e^{ikx} dt'
\label{smooth4}
\end{align}

\noi
for $t \in [0, T]$.

Therefore, from \eqref{NF1}, \eqref{smooth3}, and \eqref{smooth4},  we have
\begin{align*}
v(t) 
& = \sum_{k \in \Z_0} v_k(t) e^{ikx} \\
& = \sum_{k \in \Z_0} \bigg(\phi_k 
+  \sum_{n=2}^\infty   \N_k^n(v(t))-  \sum_{n=2}^\infty   \N_k^n(\phi)
  + \int_0^t \sum_{n=2}^\infty  \RR_k^n(v(t')) dt'
\bigg)e^{ikx}\\
& = \phi + 
 \sum_{n=2}^\infty   \N^n(v(t))-  \sum_{n=2}^\infty   \N^n(\phi)
 + \int_0^t  \sum_{n=2}^\infty   \RR^n(v(t')) dt'
\end{align*}

\noi
for $t \in [0,T]$. 
This completes the proof.
\end{proof}

In the remaining part of this section, 
we consider the normal form equation \eqref{NF2}.
In particular, 
we prove
local well-posedness of \eqref{NF2}  in the Fourier-Lebesgue spaces $\F L_0^{s,p}$
for small initial data.  See Proposition \ref{PROP:NFGWP} below.
 This will be achieved  by the contraction mapping principle. 
  In the following lemma, we establish the key multilinear estimates of $\N^n(v) $ in $\F L_0^{s,p}$.

\begin{lemma} \label{LEM:FL}
Assume that  $(s, p)$ satisfy  \textup{(i)} $s>-\frac{1}{p}, \  p> 1$ or 
\textup{(ii)} $s\geq-1, \ p=1$. 
Then, for each integer $n \geq 2$, there exists $ C_s(n) > 0$ such that 
\begin{align}
\| \N^n(v) \|_{\F L_0^{s,p}} \leq C_s(n)  \frac{Z_{s,p}^{n-1}}{2^{n-1}(n-1)!} \,\|v\|_{\F L_0^{s,p}}^n 
\label{FL1}
 \end{align}
and
\begin{align} 
\hspace{-2mm}
\| \N^n(v) - \N^n(\widetilde{v}) \|_{\F L_0^{s,p}} 
\leq  C_s(n) \frac{n Z_{s,p}^{n-1}}{2^{n-1}(n-1)!} \big( \|v\|_{\F L_0^{s,p}} + \|\widetilde{v}\|_{\F L_0^{s,p}} \big)^{n-1} \| v-\widetilde{v} \|_{\F L_0^{s,p}}, 
\label{FL2}
 \end{align}
 
\noi
where $C_s(n)$ satisfies
\begin{align*}
C_s(n) = 
\begin{cases}
1, & \text{if } s\leq 0, \\
n^s, & \text{if } s> 0.
\end{cases}	
\end{align*}

\end{lemma}

\begin{proof}
Recall the following estimate:
\begin{align}
\bigg(\sum_{j = 1}^n a_j \bigg)^\theta
\leq 
\begin{cases}
\sum_{j = 1}^n a_j ^\theta, & \text{if } 0 \leq \theta \leq 1, \\
n^{\theta-1}\sum_{j = 1}^n a_j ^\theta, & \text{if } 1 \leq \theta < \infty, \\
\end{cases}	
\label{FL3}
\end{align}
	
\noi
where the first estimate follows from the concavity of $x \mapsto x^\theta$
and the second from H\"older's inequality.
Let $w_k := |k|^s v_k$ so that  $\| v \|_{\F L_0^{s,p}} = \| w_k \|_{\ell^p}$.
Then, 
by \eqref{rec3} and \eqref{FL3}
followed by Young's and H\"older's inequalities,  we have 
\begin{align}
\| \N^n(v) \|_{\F L_0^{s,p}} 
&\les_n  \bigg\| |k|^s k \sum_{|\bk|_n = k}
e^{i \Phi_n(\bk)t}
 \prod_{j=1}^n \frac{v_{k_j}}{k_j} \bigg\|_{\ell^p} 
\leq  \bigg\||k|^{s+1}  \sum_{|\bk|_n=k}  \prod_{j=1}^n \frac{|w_{k_j}|}{|k_j|^{s+1}} \bigg\|_{\ell^p}\notag \\
&\leq C_s(n) 
\sum_{\l=1}^n \bigg\| \sum_{|\bk|_n=k} |w_{k_\l}| \prod_{j \not= \l} \frac{|w_{k_j}|}{|k_j|^{s+1}} \bigg\|_{\ell^p} 
\leq C_s(n) n   \| w_{k} \|_{\ell^p}  \bigg\| \frac{w_{k}}{|k|^{s+1}}\bigg\|_{\ell^1}^{n-1}   \notag \\
&\leq C_s(n) n Z_{s,p}^{n-1} \|w_k\|_{\ell^p}^n.
\label{FL4}
\end{align}

\noi
Noting that the omitted constant in the first inequality is $1/(2^{n-1} \,n!)$, 
we obtain \eqref{FL1}.
The second estimate \eqref{FL2} follows from a similar argument
with the telescoping sum:
\[ \prod_{j=1}^n w_{k_j} - \prod_{j=1}^n \widetilde{w}_{k_j} 
= \sum_{i=1}^n w_{k_1} \cdots w_{k_{j-1}} \cdot (w_{k_j} - \widetilde{w}_{k_j})
 \cdot \wt w_{k_{j+1}} \cdots \wt w_{k_n}.
\]

\noi 
This completes the proof.
\end{proof}

We conclude this section 
by proving small data local well-posedness of the normal form equation \eqref{NF2}.

\begin{proposition}\label{PROP:NFGWP}
Suppose that  $(s, p)$ satisfy  \textup{(i)} $s>\frac 12 -\frac{1}{p}, \  p> 2$ or 
\textup{(ii)} $s\geq 0, \ p=2$. 
Then, there exists a constant $\dl_1= \dl_1 (s, p) > 0$ such that if $\phi \in \F L_0^{s,p}(\T)$ satisfies
\begin{align}
  \|\phi\|_{\F L_0^{s,p}} \leq  \dl_1   , 
\label{NFG0}
  \end{align}

\noi
then there exist $T = T(\|\phi\|_{\F L^{s, p}_0})>0$
and  a unique  solution $v$
to the normal form equation \eqref{NF2} in $C([0, T];\F L_0^{s,p}(\T))$ 
with  $v|_{t = 0} = \phi$.
\end{proposition}

\begin{proof}
{\bf Part 1:} We first construct a  local weak solution in $L^\infty([0, T];\F L_0^{s,p}(\T))$.\footnote{
Here,  by a ``weak'' solution, 
we mean a solution without   continuity in time.}
 Given $\phi \in \F L_0^{s,p}(\T)$,  define a map $\G= \G_\phi$ on $L^\infty([0, T];\F L_0^{s,p}(\T))$ by
\begin{align}\G[v](t) := \phi + \sum_{n=2}^\infty \big(\N^n(t)(v(t))-\N^n(t)(\phi)\big)
+ \int_0^t \sum_{n=2}^\infty \RR^n(t')(v(t'))dt'.
\label{NFWP1b}
\end{align}

Define $A_s > 0$ to be a constant 
such that 
\begin{align}
 \frac{C_s(n) n}{2^{n-1}}  \leq A_s
\label{NFWP1a}
 \end{align}
 
 \noi
  for all $n \in \mathbb{N}$.
Let $B^s_R = B^s_R(T)$ be the closed ball in $X^s(T) : = L^\infty([0, T];\F L_0^{s,p}(\T))$
of radius $R
= 2 \|\phi\|_{\F L_0^{s,p}}$ centered at the origin.
In addition, we assume that $\|\phi\|_{\F L_0^{s,p}}$
is sufficiently small such that 
\begin{align}
A_s(e^{2Z_{s, p}R} - 1) \leq \frac 14
\label{NFWP1}
\end{align}

\noi
Lastly, we choose $T = T(R) = T(\|\phi\|_{\F L^{s, p}_0}) > 0$
such that 
\begin{align}
T 
= \min\bigg(\frac{ 1}{2A_s z_{s, p}^2 R^2 e^{Z_{s, p}R} }, 
\frac{1}{A_s z_{s, p}^2 R^2 e^{Z_{s,p}R} \big(2   +  e^{Z_{s,p}R} \big)}\bigg).
\label{NFWP2}
\end{align}

Let $v, \wt v \in B^s_R$.
Then, from Lemma \ref{LEM:FL} with 
\eqref{rec3a}, \eqref{embed}, \eqref{NFWP1},  and \eqref{NFWP2}, we have 
\begin{align}
\| \G[v]  \|_{X^s(T)} 
& \leq \|\phi\|_{\F L_0^{s,p}} 
+ \sum_{n=2}^\infty \big( \|\N^n(v(t))\|_{X^s(T)} + \|\N^n(\phi)\|_{X^s(T)} \big) 
\notag \\
& \hphantom{XXlllllXX}
+ \frac 14Tz_{s, p}^2 \| v \|_{ X^s(T)}^2 
 \sum_{n=1}^\infty  \|\N^n(v)\|_{X^s(T)}  \notag \\
 &\leq \|\phi\|_{\F L_0^{s,p}} 
+  \sum_{n=2}^\infty \frac{C_s(n)}{2^{n-1}} \frac{Z_{s,p}^{n-1}}{(n-1)!} 
\big( \|v\|_{X^s(T)}^n + \|\phi\|_{\F L_0^{s,p}}^n \big) \notag \\
& \hphantom{XXlllllXX}
+ \frac 14 T 
z_{s, p}^2
\| v \|_{X^s(T)}^2
 \sum_{n=1}^\infty 
 \frac{C_s(n)}{2^{n-1}} \frac{Z_{s,p}^{n-1}}{(n-1)!} \|v\|_{X^s(T)}^n\notag \\
&=  
 \|\phi\|_{\F L_0^{s,p}} \Big\{
  A_s \Big(e^{Z_{s,p} \|\phi\|_{\F L_0^{s,p}} } - 1 \Big) + 1\Big\}
+ A_s \|v\|_{X^s(T)} \Big( e^{ Z_{s,p} \|v\|_{X^s(T)} } - 1 \Big) \notag\\
& \hphantom{XXlllllXX}
+ \frac 14 T A_s z_{s, p}^2 
\|v\|_{X^s(T)}^3
e^{Z_{s,p} \|v\|_{X^s(T)}}\notag\\
& \leq  R. 
\label{NFWP3}
\end{align}

\noi
Similarly, we have
\begin{align}
\big\| \G[v]- &\G[\wt v ] \big\|_{X^s(T)}
\leq \sum_{n=2}^\infty \big\| \N^n(v) - \N^n(\wt v) \big\|_{X^s(T)} 
+ T \sum_{n=2}^\infty \| \RR^n(v)- \RR^n(\wt v)\|_{X^s(T)}
\notag\\
&\leq A_s \sum_{n=2}^\infty \frac{Z_{s,p}^{n-1}}{(n-1)!} 
\big(\|v\|_{X^s(T)}+\|\wt v \|_{X^s(T)}\big)^{n-1} \|v-\wt v\|_{X^s(T)}\notag \\
& + 
\frac 14 T  A_s z_{s, p}^2
\sum_{n=1}^\infty \frac{Z_{s,p}^{n-1}}{(n-1)!} \|v\|_{X^s(T)}^n
\big(\| v\|_{X^s(T)} +\| \wt v\|_{X^s(T)}\big)\|v - \wt v\|_{X^s(T)} \notag\\
& 
+ 
\frac 14 T  A_s z_{s, p}^2\| \wt v\|_{X^s(T)}^2  \sum_{n=1}^\infty \frac{Z_{s,p}^{n-1}}{(n-1)!} 
\big(\|v\|_{X^s(T)}+\|\wt v \|_{X^s(T)}\big)^{n-1} \|v-\wt v\|_{X^s(T)}\notag \\
&\leq A_s \big( e^{2Z_{s,p}R} - 1\big) 
 \|v-\wt v\|_{X^s(T)}
+ \frac 14 T  A_s z_{s, p}^2
\big(2R^2  e^{Z_{s,p}R} + R^2  e^{2Z_{s,p}R} \big) \|v-\wt v\|_{X^s(T)}\notag \\
&\leq  \frac 12  \|v-\wt v\|_{X^s(T)}.
\label{NFWP4}
\end{align}

\noi
Hence, it follows from 
 \eqref{NFWP3} and \eqref{NFWP4} 
that $\G$ is a contraction on $B^s_R$. 
Therefore, there exists a unique weak solution 
$v\in L^\infty([0, T];\F L_0^{s,p}(\T))$ to \eqref{NF2}
as long as $\|\phi\|_{\F L_0^{s,p}}  \leq \dl_1 \ll 1$, 
satisfying \eqref{NFWP1}.

Note that the argument above does not yield 
 continuity in time of the solution $v\in L^\infty([0, T];\F L_0^{s,p}(\T))$.
This is due to the presence of the boundary terms $\N^n(t)(v(t))$ 
without time integration
in the definition \eqref{NFWP1b} of $\G[v]$.
In the following, we will present an additional argument to show the fixed point $ v$ indeed lies in 
$C([0, T];\F L_0^{s,p}(\T))$.

\smallskip

\noi
{\bf Part 2:}
Next, we prove the persistence of regularity.
We now suppose that $\phi \in \F L^{s+\s, p}_0(\T)$
for some $\s > 0$.
Let $v \in B^s_R \subset L^\infty([0, T]; \F L_0^{s, p})$ be 
the solution to \eqref{NF2}
with $v|_{t = 0} = \phi$
constructed in Part 1,
where $T = T(\|\phi\|_{\F L^{s, p}_0})$ is chosen as in \eqref{NFWP2}.

Proceeding as in \eqref{FL4} with $w_k = |k|^s v_k$, we have
\begin{align}
\| \N^n(v) \|_{\F L_0^{s+\s ,p}} 
&\leq \frac{1}{2^{n-1}n!}
  \bigg\||k|^{s+\s+ 1}  \sum_{|\bk|_n=k}  \prod_{j=1}^n \frac{|w_{k_j}|}{|k_j|^{s+1}} \bigg\|_{\ell^p}\notag \\
&\leq \frac{C_{s+\s} (n) n Z_{s,p}^{n-1}}{2^{n-1}n!} \|v\|_{\F L^{s, p}_0}^{n-1}
\|v\|_{\F L^{s+\s, p}_0}.
\label{NFWP5}
\end{align}

\noi
Then, 
proceeding as in \eqref{NFWP3}
with  \eqref{NFWP5}, 
 we have 
\begin{align}
\| v  \|_{X^{s+\s}(T)} 
 &\leq \|\phi\|_{\F L_0^{s+\s,p}} \notag\\
& \hphantom{XX}
+  \sum_{n=2}^\infty \frac{C_{s+\s}(n)}{2^{n-1}} \frac{Z_{s,p}^{n-1}}{(n-1)!} 
\big( \|v\|_{X^s(T)}^{n-1}\|v\|_{X^{s+\s}(T)} 
+ \|\phi\|_{\F L_0^{s,p}}^{n-1}\|\phi\|_{\F L_0^{s+\s,p}} \big) \notag \\
& \hphantom{XX}
+\frac 14 T 
z_{s, p}^2
 \sum_{n=1}^\infty 
 \frac{C_{s+\s}(n)}{2^{n-1}} \frac{Z_{s,p}^{n-1}}{(n-1)!} 
  \|v\|_{X^s(T)}^{n+1}\|v\|_{X^{s+\s}(T)} \notag \\
&=  
\|\phi\|_{\F L_0^{s+\s,p}} 
\Big\{
  A_{s+\s} \Big(e^{Z_{s,p} \|\phi\|_{\F L_0^{s,p}} } -  1 \Big) + 1\Big\}
+ A_{s+\s} \|v\|_{X^{s+\s}(T)} \Big( e^{ Z_{s,p} \|v\|_{X^s(T)} } - 1 \Big) \notag\\
& \hphantom{XX}
+ \frac 14 T A_{s+\s} 
z_{s, p}^2 \| v\|_{X^s(T)}^2
e^{Z_{s,p} \|v\|_{X^s(T)}}
\|v\|_{X^{s+\s}(T)}, 
\label{NFWP6}
\end{align}

\noi
where $A_{s+\s}$ is as in \eqref{NFWP1a}.
Given $\s > 0$, 
we choose
$\dl(s, \s) >0$ sufficiently small 
such that  $R = 2 \| \phi\|_{\F L ^{s, p}_0} $
with $ \| \phi\|_{\F L ^{s, p}_0} \leq \dl(s, \s) $
 satisfies 
\begin{align}
A_{s+\s} (e^{Z_{s, p}R} - 1) \leq \frac 14.
\label{NFWP7}
\end{align}

\noi
If necessary, we make $T$ smaller such that 
\begin{align}
T \leq  \frac{ 1}{A_{s+\s} z_{s, p}^2 R^2 e^{Z_{s, p}R} }, 
\label{NFWP8}
\end{align}

\noi
Hence, from \eqref{NFWP6}, \eqref{NFWP7}, and \eqref{NFWP8}, 
we obtain 
\begin{align*}
\| v  \|_{X^{s+\s}(T)} 
&\leq  
\frac {5}{2}\|\phi\|_{\F L_0^{s+\s,p}}. 
\end{align*}

\noi
Therefore, we conclude that 
 $v \in  L^\infty([0, T]; \F L^{s+\s, p}_0)$.
 The important point is that, while $T = T(s, \s)$ now depends on
$ \|\phi\|_{\F L_0^{s,p}}$
 and  
  $\s$, 
 it is independent of 
$\|\phi\|_{\F L_0^{s+\s,p}}$.

\medskip

\noi
{\bf Part 3:}
Lastly, we prove that the solution constructed in Part 1 is indeed
in  $C([0, T];\F L_0^{s,p}(\T))$. 
We first prove continuity in 
$\F L_0^{s,p}(\T)$ for smoother solutions.
Given small $\phi \in \F L_0^{s+1,p}(\T)$, 
let $v \in B_{R}^s \subset L^\infty([0, T];\F L_0^{s,p}(\T))$
be the weak solution with $v|_{t = 0} = \phi$ constructed in Part 2, 
where $T = T(s, \s)$ with $\s = 1$.
Note that from Part 2, we have
\begin{align}
\|v\|_{L^\infty([0, T]; \F L^{s+1, p}_0)} \leq 
R_{s+1}: = R_{s+1} (\|\phi\|_{\F L^{s+1, p}_0}).
\label{NFWP8b}
\end{align}

Write \eqref{NF2} as 
\begin{equation} 
v(t) = \phi + \N (t) +  \RR (t), 
\label{NFWP8a} 
\end{equation}

\noi
where 
\[ \N(t) = \sum_{n=2}^\infty  \N^n(v(t)) - \sum_{n=2}^\infty\N^n(\phi) 
\quad \text{and} \quad
\RR (t)   = \int_0^t \sum_{n=2}^\infty \RR^n(v(t'))dt'.\]

\noi
Thanks to the time integral, 
using Lemma \ref{LEM:FL} with \eqref{rec3a}, 
one can easily show
that  $\RR(t) \in C([0, T]; \F L_0^{s, p}(\T))$.
Indeed, 
given $t_1, t_2 \in \R$, 
we have 
\begin{align}
\|\RR(t_1) - \RR(t_2) \|_{\F L^{s, p}_0} 
\leq  \frac 14 A_s z_{s, p}^2 
R^3
e^{Z_{s,p} R}
|t_1-t_2|.
\end{align}

By the triangle inequality, we have
\begin{align}
\| \N^n(t_1)(v(t_1)) - \N^n(t_2)(v(t_2)) \|_{\F L_0^{s,p}} 
&\leq 
\| \N^n(t_1)(v(t_1)) - \N^n(t_2)(v(t_1)) \|_{\F L_0^{s,p}} \notag\\
& \hphantom{X}
+ 
\| \N^n(t_2)(v(t_1)) - \N^n(t_2)(v(t_2)) \|_{\F L_0^{s,p}} \notag\\
& =: \1_n + \2_n.
\end{align}

\noi
Arguing as in \eqref{NFWP4} 
with \eqref{NFWP1}, 
we have 
\begin{align}
\sum_{n = 2}^\infty\2_n
\leq A_s \sum_{n = 2}^\infty \frac{ (2Z_{s,p}R)^{n-1}}{(n-1)!} 
 \| v(t_1) - v(t_2) \|_{\F L_0^{s,p}} 
\leq \frac 14
 \| v(t_1) - v(t_2) \|_{\F L_0^{s,p}}. 
\end{align}

\noi
By Mean Value Theorem with $w_k := |k|^s v_k$, 
\eqref{Phi}, and \eqref{FL3}, 
we have
\begin{align*}
\1_n 
& 
\leq  \frac{1}{2^{n-1}n!}\bigg\||k|^{s+1}  \sum_{|\bk|_n=k} 
|\Phi(\bk)(t_1 - t_2)|
 \prod_{\l=1}^n \frac{|w_{k_\l}(t_1) |}{|k_\l|^{s+1}} \bigg\|_{\ell^p} \notag\\
& \le
|t_1 - t_2|
\frac{1}{2^{n-2}n!}\bigg\||k|^{s+1}  
\sum_{1 \leq i < j \leq n}
\sum_{|\bk|_n=k} 
|k_i| |k_j|
 \prod_{\l=1}^n \frac{|w_{k_\l}(t_1) |}{|k_\l|^{s+1}} \bigg\|_{\ell^p} \notag \\
&\leq 
|t_1 - t_2|
\frac{C_s(n)  n }{2^{n-1}(n-2)!}
 \big\| |k| w_{k} (t_1)\big\|_{\ell^p}  \bigg\| \frac{w_{k}(t_1)}{|k|^{s}}\bigg\|_{\l^1} \bigg\| \frac{w_{k}(t_1)}{|k|^{s+1}}\bigg\|_{\ell^1}^{n-2}   \notag \\
&\leq 
|t_1 - t_2|
\frac{C_s(n)  n Z_{s,p}^{n-1}}{2^{n-1}(n-2)!}
 \| v(t_1)  \|_{\F L_0^{s+1,p}}^2 
 \| v(t_1)  \|_{\F L_0^{s,p}}^{n-2}.
\end{align*}

\noi
Choosing  $B_s > 0$ such that 
$2^{-(n-1)} C_s(n) n Z_{s, p}  \leq B_s$ for all $n \in \mathbb{N}$, 
and recalling \eqref{NFWP8b}, 
we have 
\begin{align}
\sum_{n = 2}^\infty \1_n 
&\leq 
B_s R_{s + 1}^2 e^{Z_{s, p} R}
|t_1 - t_2|
\label{NFWP12} 
\end{align}

\noi
Hence, from \eqref{NFWP8a} -- \eqref{NFWP12}
we have
\begin{align*} 
\| v(t_1) - v(t_2) \|_{\F L_0^{s,p}} 
&\les 
 B_s R_{s + 1}^2 e^{Z_{s, p} R}
|t_1 - t_2|
+   A_s z_{s, p}^2 
R^3
e^{Z_{s,p} R}
|t_1-t_2|.
\end{align*}

\noi
Therefore, we conclude that 
 $v \in C(\R;\F L_0^{s,p}(\T))$.

Now, given small $\phi \in \F L_0^{s,p}(\T)$, 
let $v \in L^\infty([0, T];\F L_0^{s,p}(\T))$
be the weak solution with $v|_{t = 0} = \phi$ constructed in Part 2, 
where $T = T(s, \s)$ with $\s = 1$.
Let 
$\{\phi^{(j)}\}_{j \in \mathbb{N}} \subset \F L_0^{s+1,p}(\T)$
such that $\phi^{(j)}$ converges to $\phi$ in $\F L_0^{s,p}(\T)$, 
as $j \to \infty$.
Denote by $v^{(j)}$
the corresponding global solutions 
in $  L^\infty([0, T];\F L_0^{s+1,p}(\T))\cap  C([0, T];\F L_0^{s,p}(\T))$.
Recalling that 
$T = T(s, \s)$ depends only on $\| \phi^{(j)}\|_{\F L^{s, p}_0}$
and $\s > 0$ but is independent of $\| \phi^{(j)}\|_{\F L^{s+1, p}_0}$, 
we can choose 
$T = T(s, \s)$ independent of $j \in \mathbb{N}$.
Then, arguing as in \eqref{NFWP4}, we have 
\begin{align*}
\| v - v^{(j)}\|_{X^s(T)}
\leq 2 \| \phi - \phi^{(j)} \|_{\F L^{s, p}_0}
\end{align*}
	
\noi
Therefore, 
 $v^{(j)}(t)$ converges to $v(t)$ in the $\F L_0^{s,p}$-topology,  uniformly in $t\in [0, T]$.
Hence, as a uniform limit of continuous functions, 
we conclude that 
$v\in  C([0, T];\F L_0^{s,p}(\T))$.
\end{proof}

\begin{remark}[Unconditional uniqueness]\label{REM:uniq}\rm
Given $ \phi \in \F L^ {s, p}_0$ satisfying \eqref{NFG0}, 
suppose that $v, \wt v \in C([0, T]; \F L^ {s, p}_0)$
are two solutions to \eqref{NF2} with $v|_{ t= 0} =  \wt v|_{t = 0} = \phi$.
Without loss of generality, assume that $v$ is the solution constructed in Proposition
\ref{PROP:NFGWP}.
In particular, we have $v \in B_R^s(T)$ with $R =  2\|\phi\|_{\F L^ {s, p}_0}$.
In the following, redefine the local existence time $T$ in \eqref{NFWP2} and \eqref{NFWP8}
by replacing $R$ with $4R$.

In general, we may have $\wt v \notin  B_R^s(T)$.
By the continuity in time with $\wt v |_{t = 0} = \phi$, 
there exists small $\tau_1 > 0$ such that 
$\sup_{t \in [0, \tau_1]} \| \wt v(t)\|_{\F L^{s, p}_0} \leq  2 R.$
Then, it follows from from \eqref{NFWP4} that 
$v = \wt v$ in  $C([0, \tau_1]; \F L^{s, p}_0)$.
Since $\|\wt  v(\tau_1)\|_{\F L^{s, p}_0} = \| v(\tau_1)\|_{\F L^{s, p}_0} \leq R$, 
there exists small $\tau_2 > 0$ such that 
$\sup_{t \in [\tau_1, \tau_2]} \| \wt v(t)\|_{\F L^{s, p}_0} \leq  2 R.$
Then, it follows from from \eqref{NFWP4} that 
$v = \wt v$ in  $C([\tau_1, \tau_2]; \F L^{s, p}_0)$
and conclude that 
 $\|\wt  v(\tau_2)\|_{\F L^{s, p}_0}  \leq R$.
In this way, we can cover the entire interval $[0, T]$
and conclude that $v = \wt v $ in $C([0, T]; \F L^{s, p}_0)$.
Namely, we have unconditional uniqueness.
\end{remark}

\begin{remark}\label{REM:pers}\rm
In Part 2 of the proof of Proposition \ref{PROP:NFGWP}, 
we proved the persistence of regularity 
only for a fixed value of $p \geq 2$.
In general, the persistence of regularity
also holds for different values of $p$
as long as a proper embedding holds.
Let $B^{s, p}_R = B^{s, p}_R(T)$ be the closed ball in $X^{s, p}(T) : = L^\infty([0, T];\F L_0^{s,p}(\T))$
in the following.

Suppose that $\phi \in \F L^{s, p}_0(\T)$
with $s > \frac 12 - \frac 1p$ such that 
$ \F L^{s, p}_0(\T) \subset L^2(\T)$.
Let $v \in B_R^{0, 2} \subset X^{0, 2}(T)$ be 
the solution to \eqref{NF2}
with $v|_{t = 0} = \phi$
constructed in Part 1, 
where $T = T(\|\phi\|_{L^2})$ is chosen as in \eqref{NFWP2}.
Proceeding as in \eqref{FL4} and \eqref{NFWP5} with $w_k = |k|^s v_k$, we have
\begin{align}
\| \N^n(v) \|_{\F L_0^{s ,p}} 
&\leq \frac{C_{s} (n) n Z_{0,2}^{n-1}}{2^{n-1}n!} \|v\|_{L^2}^{n-1}
\|v\|_{\F L^{s, p}_0}.
\label{pers1}
\end{align}

\noi
Then, a computation analogous to \eqref{NFWP6} with \eqref{pers1}
yields
\begin{align}
\| v  \|_{X^{s, p}(T)} 
& \leq  
\|\phi\|_{\F L_0^{s,p}} 
\Big\{
  A_{s} \Big(e^{Z_{0,2} \|\phi\|_{L^2 }} -  1 \Big) + 1\Big\}
+ A_{s} \|v\|_{X^{s, p}(T)} \Big( e^{ Z_{0,2} \|v\|_{X^{0, 2}(T)} } - 1 \Big) \notag\\
& \hphantom{XX}
+ \frac 14 T A_{s} 
 \| v\|_{X^{0, 2}(T)}^2
e^{Z_{0,2} \|v\|_{X^{0, 2}(T)}}
\|v\|_{X^{s, p}(T)}. 
\label{pers2}
\end{align}

\noi
Then, by choosing $R$ and $T$ sufficiently small
such that \[A_s(e^{Z_{0, 2}R} - 1) \leq \frac 14
\qquad \text{and}\qquad  
T \leq  \frac{ 1}{A_{s}  R^2 e^{Z_{0, 2}R} }, 
\]

\noi
it follows from 
\eqref{pers2} that 
\begin{align*}
\| v  \|_{X^{s, p}(T)} 
&\leq  
 \frac 52\|\phi\|_{\F L_0^{s,p}}. 
\end{align*}

\noi
Therefore, we conclude that 
 $v$ also lies  in $  L^\infty([0, T]; \F L_0^{s, p})$.

\end{remark}

\section{Small data global existence of smooth solutions}
\label{SEC:CH}

\subsection{Cole-Hopf transformation}\label{SUBSEC:CH}

In the non-periodic setting, 
we can use the Cole-Hopf transformation \eqref{gauge1} 
to transform smooth solutions to \eqref{DNLS1a} on $\R$ 
into
solutions of the linear Schr\"odinger equation.
In the periodic case, however, 
we need to make a suitable adjustment
so that a modified Cole-Hopf transformation 
converts smooth mean-zero solutions to \eqref{DNLS1a} on $\T$
into solutions of the linear Schr\"odinger equation.

Given a mean-zero function $\phi$ on $\T$, 
we follow the Cole-Hopf transformation \eqref{gauge1} on $\R$ and 
define a gauge transformation $\GG_0$ 
by  setting
\begin{equation} \label{CH1}
  \GG_0[\phi] : = e^{-\frac{i}{2} \J(\phi)},  
\end{equation}

\noi
where $\J(\phi)$ is the mean-zero primitive of $\phi$ given by 
\begin{equation*}
 \J(\phi)_k = \begin{cases} \frac{\phi_k}{ik}, & k\ne 0,    \\ 0, & k=0. \end{cases}  
\end{equation*}

\noi
Note that $\GG_0[\phi]$ is a periodic function on $\T$, since $\J(\phi)$ is periodic.
Suppose that $u$ is a smooth mean-zero solution to dNLS \eqref{DNLS1a} on $\T$
and let $w (t) := \GG_0[u(t)]$.
Unfortunately,   such $w$ does not satisfies the linear Schr\"odinger equation.
Indeed, we have
\begin{align}
i \dt w + \dx^2 w = - \frac 14 \P_0[u^2] \cdot w, 
\label{CH2a}
\end{align}

\noi
where $\P_0[u^2]$ is defined in \eqref{mass}.

In view of \eqref{CH2a}, we define a new ``gauge'' transformation
$\GG$ on functions depending on both $x$ and $t$.
Given a smooth function $u(t, x)$ on $[0, T] \times \T$
such that $\int_\T u(t) dx  = 0$ for all $t \in [0, T]$, 
define a gauge transformation $\GG$ by 
\begin{equation} \label{CH3}
 W (t) = \GG[u](t) : = e^{ - \frac i4 \int_0^t \P_0[u^2(t')] dt'} e^{-\frac{i}{2} \J(u(t))},  
\end{equation}

\noi
In particular, we have  $\GG[u](0) = \GG_0(u(0))$.
Note that 
\begin{align}
\| \dx W(0) \|_{L^2} \leq \frac 12 \big\| e^{\frac 12 \Im \J(u(0))} \big\|_{L^\infty} \| u(0)\|_{L^2}
\leq 
\frac 12 e^{\frac 12 Z_{0, 2}\|u(0)\|_{L^2}} \| u(0)\|_{L^2}, 
\label{CH3a}
\end{align}

\noi
where $Z_{0, 2}$ is as in Definition \ref{DEF:Zsp}.

Suppose that $u(t, x)$ is a smooth mean-zero solution to \eqref{DNLS1a}
on an interval $[0, T]$.
Define $W$ by \eqref{CH3}.
Then, it is easy to see that $W $ satisfies the linear Schr\"odinger equation  on $(0, T) \times \T$:
\begin{align}
\dt W =i\dx^2 W.
\label{CH4}
 \end{align}

\noi
Given a smooth mean-zero initial condition  $\phi$, 
we can use the gauge transformation $\GG$ to 
 construct a smooth solution $u$ to \eqref{DNLS1a}.
 Indeed, set $W(0) =  \GG_0[\phi]$
and let $W(t)$ be the global solution to the linear Schr\"odinger equation \eqref{CH4}.
Then, applying the inverse gauge transformation: 
\begin{equation} \label{CH5}
u(t, x) = \GG^{-1}[W](t, x) := 2i \frac{\dx W(t, x)}{W(t, x)}, 
\end{equation}

\noi
we see that   $u$ is a solution to \eqref{DNLS1a}
as long as \eqref{CH5} makes sense.
Note that the smoothness of $\phi$ (and hence of $W(t)$)
was needed to consider the pointwise division in \eqref{CH5}.
Obviously, \eqref{CH5} makes sense as long as 
$W(t, x) \ne 0$.

\begin{figure}[h]

\begin{equation*} \xymatrix{W(0) = \GG_0[\phi]  \ar@{<-}[d]_{\eqref{CH1}, \, \eqref{CH3}} 
\ar@{->}[rrr]^{\dt W  = i \dx^2 W} 
& & &  W(t) \ar@{->}[d]^{\eqref{CH5}}\\ 
 u(0) = \phi \ar@2{->}[rrr]_{\text{dNLS}} & & & u(t) = \GG^{-1}[W](t)  }\end{equation*}
\caption{Relation between dNLS and the linear Schr{\"o}dinger equation.}
\label{Fig:1}
\end{figure}

Now, let us introduce a geometric view point.
The image
of a complex-valued periodic function is  a closed loop in $\C$. 
Thus, the inverse transformation \eqref{CH5} makes sense
at time $t$ if the loop $W(t)$ stays away from the origin in the complex plane.
Note that the trivial solution $u(t) \equiv 0$ to \eqref{DNLS1a}
is transformed into $W(t)  \equiv 1$.
Thus, a small initial condition  $u(0)$ to \eqref{DNLS1a}
are transformed into a small loop $W(0)$ around $1\in \C$.
In particular,  it is  away from the origin in the complex plane.
Then, we expect that the solution $W(t)$
to \eqref{CH4} remains as a small loop around $1 \in \C$ for all $t \in \R$, allowing us to apply the inverse
transformation \eqref{CH5}.
In the following, we make this intuition precise.
In particular, by assuming that a smooth mean-zero initial condition $u(0) = \phi$ is sufficiently small, 
we show that $W(t)$ stays away from the origin for all time, giving rise 
to a smooth global solution $u$ to \eqref{DNLS1a}.

\begin{remark} \label{REM:winding}
\rm 
Given a smooth periodic function   $u(t, x)$  on $[0, T] \times \T$
such that $\int_\T u(t) dx  = \mu \in \C$ for all $t \in [0, T]$, 
let $W(t)$ be the gauge transformation of $u(t)$ defined in \eqref{CH3}.
We need to make sure that $W(t)$ is a periodic function for each $t \in [0, T]$.
From the geometric point of view, 
$W(t)$ must be a closed loop in $\C$ for each $t \in [0, T]$.
In particular, the {\it index} (= winding number) of the loop $\g = W(t)$ at the origin
must be well defined.
By a direct computation, 
we have
\[ \Ind_{\g}(0) = \frac{1}{2\pi i} \int_{\g} \frac{dz}{z} 
= \frac{1}{2\pi i} \int_{\T} \frac{\dx W(t)}{W(t)}dx 
= - \frac{1}{4\pi} \int_\T u(t)dx = - \frac{\mu}{4\pi}, \]

\noi
where the third equality follows from \eqref{CH5}.
Hence, we must have $\mu \in 4 \pi \Z$.

In this paper, we only consider the mean-zero functions $u$, 
corresponding to the loop $W$ of index 0 at the origin.
It may be of interest to study well-posedness of \eqref{DNLS1a}
with $\int_\T u(0) dx = - 4\pi m$, $m \in \Z$, 
corresponding to the loop $W(0)$ of index $m$ at the origin.

\end{remark}

\begin{remark}\label{REM:linear}\rm

The normal form reduction performed in Section \ref{SEC:NF} corresponds
to the Taylor expansion of (the derivative of
the interaction representation of) the gauge transformation $W$ defined in \eqref{CH3}.
Indeed, we have 
\begin{align}
\dx S(-t) W (t) 
& =  e^{ - \frac i4 \int_0^t \P_0[u^2(t')] dt'}
\dx S(-t) e^{-\frac{i}{2} \J(u(t))} \notag \\  
& =  e^{ - \frac i4 \int_0^t \P_0[u^2(t')] dt'}
\dx  S(-t) \bigg[ \sum_{n=0}^\infty \frac{1}{n!} \bigg( -\frac i2 \J(u(t)) \bigg)^n   \bigg] \notag \\
& =  e^{ - \frac i4 \int_0^t \P_0[u^2(t')] dt'}
\dx  \sum_{n=1}^\infty \mathcal{M}^n(u(t)).
\label{Tay1}
\end{align}

\noi
Here,  
$ \mathcal{M}^n(u(t))$, $n \geq 1$,  is given by 
\begin{align}
\big(\mathcal{M}^{n}(u(t))\big)_k 
&= \frac{1}{n!} \bigg(-\frac 12 \bigg)^{n} \sum_{|\bk|_n = k} e^{i\Phi_n(\bk)t} \prod_{j=1}^{n}
\frac{v_{k_j}(t)}{k_j}
= \frac{1}{2k}\N^n_k(v(t)), 
\label{Tay2}
\end{align}

\noi
where $v$ denotes the interaction representation of $u$ defined in \eqref{IR1}
and the last equality follows from \eqref{rec3} and \eqref{rec3b}.
Then, from \eqref{Tay1}, \eqref{Tay2} and \eqref{NF4}, indeed we have
\begin{align}
\big(\dx S(-t) W (t) \big)_k 
& =  \frac{i}{2}  e^{ - \frac i4 \int_0^t \P_0[u^2(t')] dt'}
 \sum_{n=1}^\infty \mathcal{N}^n_k(v(t))
 = \frac i 2 \Q_k(t)
\label{Tay3}
\end{align}

\noi
for $k \in \Z_0$.
Since $W$ is a solution to the linear Schr\"odinger equation, 
$\big(S(-t)W(t)\big)_k$ is conserved under the dynamics.
This fact can be also seen from \eqref{Tay3} and Remark \ref{REM:cons}:
\begin{align*}
\big( S(-t) W (t) \big)_k 
 = \frac 1 {2k} \Q_k(t)
 = \frac 1 {2k} \Q_k(0)  = W_k (0).
\end{align*}

\end{remark}

\subsection{Global existence of smooth solutions}
\label{SUBSEC:GWP1}

In this subsection, we prove global existence of smooth solutions
to \eqref{DNLS1a} with small mean-zero initial data. 
As mentioned in the previous subsection, 
the main goal is to 
make sure that the loop $W(t)$ in the complex plane does not intersect
the origin for any $t \in \R$.
The following simple lemma provides a sufficient condition.

\begin{lemma} \label{LEM:noint}
Suppose that $W^0 \in H^{\frac 12 + \eps } (\T)$, $\eps > 0$,  satisfies
\begin{equation}
  |W^0_0|
- \sum_{k \ne 0} |W^0_k| \geq \dl > 0, 
\label{noint1}
\end{equation}

\noi
for some $\dl > 0$.
Here,  $W^0_k$ is the $k$-th Fourier coefficient of $W^0$. 
Then, when viewed as a loop in the complex plane, the solution $W(t)$ 
to the linear Schr\"odinger equation \eqref{CH4} 
with $W|_{t = 0} = W^0$
never intersect the origin for any $t \in \R$.
Furthermore, we have 
\[ |W(t, x)| \geq \dl > 0.\]
\end{lemma}

\begin{proof}
By Sobolev embedding 
and the unitarity of the linear Schr{\"o}dinger flow, 
we have $W(t, x) \in C(\R\times \T)$.
Note that the solution $W(t)$ to the linear Schr{\"o}dinger equation  \eqref{CH4} is given by 
\[ W(t, x)  = \sum_{k \in \Z} W^0_k \, e^{ik x} \, e^{-ik^2 t}. \]

\noi
Therefore,  from \eqref{noint1}, we have
\[ |W(t, x)| = 
\bigg| W^0_0 + \sum_{k \ne 0} W^0_k  e^{ik x}  e^{-ik^2 t} \bigg| 
\geq |W^0_0| - \sum_{k \not= 0} |W^0_k| \geq \dl > 0, \]

\noi
for all $(t, x) \in \R\times \T$.
\end{proof}

\begin{remark}\rm
The condition \eqref{noint1} is sharp.
Consider 
\[ W^0(x) = - 2 \zeta(2) + \sum_{k \ne 0} \frac{1}{k^2}e^{ikx},\]

\noi
where $\zeta(\tau) = \sum_{k = 1}^\infty k^{-\tau}$ is the Riemann zeta function.
Then, clearly \eqref{noint1} is violated.
Moreover, we have $W^0(0) = 0$.

\end{remark}

As a corollary of Lemma \ref{LEM:noint}, we obtain 
the following a priori bound on the $L^2$-norm
of smooth solutions to \eqref{DNLS1a}.

\begin{lemma} \label{COR:L2}

Suppose that $u$ is a smooth global solution 
to \eqref{DNLS1a} with 
a mean-zero initial condition:  $\int_\T u(0) \, dx = 0$.
Let $W(t) = \GG[u](t)$ be the gauge transformation defined in \eqref{CH3}.
If $W^0 = W(0)$ satisfies \eqref{noint1} for some $\dl > 0$, then 
we have the following a priori bound:
\begin{align*}
\| u (t) \|_{L^2(\T)} 
\leq 
\frac{1}{\dl}
e^{\frac 12 Z_{0, 2}\|u(0)\|_{L^2}} \| u(0)\|_{L^2}
\end{align*}

\noi
for all $t \in \R$.

\end{lemma}

\begin{proof}
By Lemma \ref{LEM:noint}, 
 we have
 $|W(t, x)| \geq \dl > 0$
for all $(t, x) \in \R\times \T$.
Then, from \eqref{CH5}, the unitary of the linear Schr\"odinger flow on $\dot H^1$,  and \eqref{CH3a}, we have 
\begin{align*}
\| u (t) \|_{L^2(\T)} \leq \frac{2}{\dl} \| \dx W(t) \|_{L^2(\T)}
= \frac{2}{\dl} \| \dx W(0) \|_{L^2(\T)}
\leq  \frac{1}{\dl} 
e^{\frac 12 Z_{0, 2}\|u(0)\|_{L^2}} \| u(0)\|_{L^2}.
\end{align*}
\end{proof}

Given a smooth mean-zero function $\phi$ on $\T$, 
let $W$ be the solution to \eqref{CH4}
with $W(0) = W^0: = \GG_0[\phi]$.
In the following proposition, 
 we transfer the condition \eqref{noint1} on $W^0$ 
to a condition on $\phi$
guaranteeing that the loop  $W(t)$ does not intersect the origin for any $t \in \R$. 
This allows us to apply the inverse transformation \eqref{CH5}
and construct a smooth global solution $u (t) = \GG^{-1}[W](t)$
to \eqref{DNLS1a}.

\begin{proposition} \label{PROP:SGWP1}
Suppose that  $(s, p)$ satisfy  \textup{(i)} $s>\frac 12 -\frac{1}{p}, \  p> 2$ or 
\textup{(ii)} $s\geq 0, \ p=2$. 
Let $\phi$ be a smooth function on $\T$,  satisfying $\int_\T \phi \, dx = 0$.
Define $M = M(\phi)$ by 
\begin{align}
 M := \sup_{x\in\T} |\J(\phi)(x)|. 
\label{SGWP1}
\end{align}

\noi
 If $M < \pi$ and
\begin{align}
 e^{\frac{1}{2}Z_{s,p} \|\phi\|_{\F L^{s,p}_0}} < 2 e^{-\frac M 2} \cos\bigg(\frac{M}{2}\bigg), 
\label{SGWP2}
\end{align}

\noi
then there exists a smooth global solution $u$ to dNLS \eqref{DNLS1a} with 
$u|_{t = 0} = \phi$.
Moreover, there exists $C(\phi) > 0$ such that we have
\begin{align}
\| u (t) \|_{L^2(\T)} 
\leq C(M, \|\phi\|_{\F L^{s, p}_0}) < \infty, 
\label{SGWP2a}
\end{align}

\noi
for all $t \in \R$.

\end{proposition}

\begin{remark} \rm
This type of small ``disturbance'' condition also appears in the corresponding problem on $\R$. 
In \cite{Stef}, Stefanov proved local well-posedness of 
the quadratic dNLS \eqref{DNLS1a} in  $H^1(\R)$, 
assuming that
 $\sup_{x\in\R} \big|\int_{-\infty}^x \phi(y)\,dy\big|$ is sufficiently small. 
 This is analogous to the condition $M < \pi$ in Proposition \ref{PROP:SGWP1}.
 \end{remark}

\begin{proof}

Let $W^0= \GG_0[\phi]$.
Then, we have 
\begin{align}
W^0(x) &= e^{-\frac{i}{2}\J(\phi)(x)} 
= \sum_{n=0}^\infty \frac{1}{n!} \bigg(\frac{-i}{2}\bigg)^n[\J(\phi)(x)]^n \notag \\
&= 
\sum_{n=0}^\infty \frac{1}{n!} \bigg(\frac{-i}{2}\bigg)^n
 \sum_{k \in \Z} \sum_{|\bk|_n = k} \prod_{j=1}^n \J(\phi)_{k_j} e^{ikx} 
=: \sum_{k \in \Z} W^0_k \, e^{ikx},
\label{SGWP3}
\end{align}

\noi
where
\[ W^0_k = 
\sum_{n=0}^\infty \frac{1}{n!} \bigg(\frac{-i}{2}\bigg)^n
 \sum_{|\bk|_n = k} \prod_{j=1}^n \J(\phi)_{k_j}. \]

\noi
Then, by Young's inequality and 
\begin{align}
\|\J(\phi)_k\|_{\ell^1}
= \sum_{k \not= 0} \frac{|\phi_k|}{|k|} 
\leq Z_{s,p}  \|\phi\|_{\F L_0^{s,p}},
\label{SGWP3a}
\end{align} 

\noi
we have
\begin{align}
\sum_{k \in \Z} |W^0_k| &\leq \sum_{n=0}^\infty \frac{1}{2^n n!} 
\sum_{k \in \Z} \sum_{|\bk|_n = k} \prod_{j=1}^n |\J(\phi)_{k_j}| 
\leq \sum_{n=0}^\infty \frac{1}{2^n n!} \| \J(\phi)_k \|^n_{\ell^1} \notag \\
&\leq \sum_{n=0}^\infty \frac{1}{n!} 
\bigg(\frac{Z_{s,p} \|\phi\|_{\F L_0^{s,p}}}{2}\bigg)^n = e^{\frac{1}{2}Z_{s,p} \|\phi\|_{\F L_0^{s,p}}}.
\label{SGWP4}
\end{align}

\noi
Note that the absolute convergence of 
the summations in $n$ and $k$ of \eqref{SGWP4}
justifies the interchange of the summations in \eqref{SGWP3}.

On the other hand, let
 $r(x) := \exp\big\{\Re(-\frac{i}{2}\J(\phi)(x)) \big\}$ 
 and $\theta(x) := \Im(-\frac{i}{2}\J(\phi)(x))$.
Then, by \eqref{SGWP1}, 
 we have
\[ W^0(x) = r(x)e^{i\theta(x)}, \qquad r(x) \geq e^{-\frac M2},  \qquad \text{and} \quad |\theta(x)| \leq \frac M 2 \]

\noi
for all $x \in \T$.
Hence, by Mean Value Theorem, 
there exists $x_0 \in \T $ such that  
\begin{align}
|W^0_0| &= \frac{1}{2\pi}\bigg| \int_{\T} W^0 dx\bigg| 
= \frac{1}{2\pi}\bigg| \int_{\T} r(x)e^{i\theta(x)} dx \bigg| 
\geq \frac{1}{2\pi}\bigg| \int_{\T} r(x)\cos\theta(x) dx \bigg| \notag \\
&= 
\big| r(x_0) \cos\theta(x_0) \big| \notag\\
&\geq e^{-\frac M2} \cos \bigg(\frac M2\bigg).
\label{SGWP5}
\end{align}

\noi
The condition $0 \leq  M  < \pi$ guarantees that the right-hand side of \eqref{SGWP5}
is positive.

In view of 
\eqref{SGWP2}, \eqref{SGWP4} and \eqref{SGWP5}, 
we see that 
the condition \eqref{noint1} in Lemma \ref{LEM:noint} is satisfied for some $\dl > 0$.
Hence, letting $W$ denote the smooth global solution 
to \eqref{CH4} with $W|_{t = 0} = W^0$, 
Lemma \ref{LEM:noint} allows us to apply the inverse transformation \eqref{CH5}
and construct a smooth global  solution $u$ to \eqref{DNLS1a}
with $u|_{t = 0} = \phi$.
Lastly, the global $L^2$-bound \eqref{SGWP2a}
follows from Lemma \ref{COR:L2}.
\end{proof}

We now transfer the conditions in Proposition \ref{PROP:SGWP1} 
to a smallness condition on the $\F L^{s, p}_0$-norm of smooth mean-zero initial data.

\begin{corollary} \label{COR:SGWP2}
Suppose that  $(s, p)$ satisfy  \textup{(i)} $s>\frac 12 -\frac{1}{p}, \  p> 2$ or 
\textup{(ii)} $s\geq 0, \ p=2$. 
Then, there exists  $ \dl_2>0$ such that
if a smooth mean-zero function $\phi$ on $\T$
satisfies 
\[  \|\phi \|_{\F L_0^{s,p}(\T)} \leq \dl_2, \]

\noi
then 
there exists a smooth global  solution $u$ to  dNLS \eqref{DNLS1a}
with $u|_{t = 0} = \phi$, satisfying \eqref{SGWP2a}.

\end{corollary}
\begin{proof}

From \eqref{SGWP3a}, we have 
\begin{align}
M = \sup_{x\in\T} |\J(\phi) (x)| \leq Z_{s,p}\|\phi\|_{\F L_0^{s,p}} \leq  Z_{s,p} \, \dl_2.
\label{SGWP6}
\end{align}

\noi
Hence,  the first condition $M < \pi$ in Proposition \ref{PROP:SGWP1}
is satisfied if $Z_{s,p} \, \dl_2 < \pi$.

Let $f(x) = 2e^{-2x} \cos x$.
Then, 
there exists unique $\al \in \big(0, \frac \pi 2\big)$
such that 
$f(\al) = 1$ since
$f$ is strictly decreasing on $[0, \frac \pi 2]$, $f(0) = 2$, and $f\big(\frac \pi 2\big) = 0$.
Now, choose $\dl_2>0$ sufficiently small such that 
$Z_{s, p} \dl_2 <2 \al$.
Then, we have 
\begin{align*}
 e^{\frac{1}{2}Z_{s,p} \|\phi\|_{\F L^{s,p}_0}} 
\leq  e^{\frac 12 Z_{s, p} \dl_2 } 
< e^{\al}
= 2 e^{-\alpha} \cos\alpha < 2 e^{-\frac M 2} \cos\big(\tfrac M2\big), 
\end{align*}

\noi
where the last inequality follows from \eqref{SGWP6} and the fact that
$g(x) = e^{-x}\cos x$ is strictly decreasing on $x \in \big[0,\frac \pi 2\big]$.
This shows that 
the second condition \eqref{SGWP2}
in Proposition \ref{PROP:SGWP1} also holds.
\end{proof}

\section{Proof of Theorem \ref{THM:main}} \label{SEC:GWP}
We are now ready to present the proof of Theorem \ref{THM:main}.
The main ingredients are 
the normal form reductions 
and small data global existence of smooth solutions via the gauge transformation 
discussed 
in Sections \ref{SEC:NF}
and \ref{SEC:CH}.

Given small $\dl_0 >0$ (to be chosen later), 
fix  $\phi \in \F L_0^{s,p}$  such that  $ \|\phi \|_{\F L_0^{s,p}} \leq\delta_0 $.
For $j \in \mathbb{N}$, 
let $\phi^{(j)} = \P_{\leq j}\phi$, 
where $\P_{\leq j}$ is 
the Dirichlet projection onto the frequencies $\{|k| \leq j\}$.
Note that we have  $\{\phi^{(j)}\}_{j \in \mathbb{N}} \subset C^\infty(\T)$, 
$\int_\T \phi^{(j)}dx = 0$,  and 
$ \|\phi^{(j)} \|_{\F L_0^{s,p}} \leq\delta_0 $
for all $j \in \mathbb{N}$.
Then, by Corollary \ref{COR:SGWP2}, 
there exist global smooth solutions $u^{(j)}$ to \eqref{DNLS1a}
with $u^{(j)}|_{t = 0} = \phi^{(j)}$, 
as long as $\dl_0 \leq \dl_2$.

We first consider the case $(s, p) = (0, 2)$.
Letting $W^{(j), 0} = \GG_0 [ \phi^{(j)}]$, 
it follows from 
\eqref{SGWP4} and  \eqref{SGWP5}
with \eqref{SGWP6}
that 
\begin{align*}
|W_0^{(j), 0}| - \sum_{k \ne 0} |W^{(j), 0}_k| 
& = 2 |W_0^{(j), 0}| - \sum_{k \in \Z} |W^{(j), 0}_k| \\
&\geq 
2 e^{-\frac 12  Z_{0, 2} \dl_0 } \cos \bigg(\frac {Z_{0,2}\dl_0} 2\bigg)
-  e^{\frac{1}{2}Z_{0,2} \dl_0 }
=: A(\dl_0) > 0
\end{align*}

\noi
for all $j \in \mathbb{N}$.
From the continuity of $A(\dl_0)$
and 
$A(0) = 1$, 
we have $A(\dl_0) \geq \frac 12$ for all sufficiently small $\dl_0 > 0$.
Hence, 
it follows from Lemma \ref{COR:L2}
that by choosing $\dl_0 > 0$ sufficiently small, 
we have 
\begin{align}
\| u^{(j)} (t) \|_{L^2} 
\leq 
\| u (t) \|_{L^2(\T)} 
\leq 
2
e^{\frac 12 Z_{0, 2}\dl_0 } \dl_0 < \dl_1, 
\label{GWP1}
\end{align}

\noi
for all $t \in \R$
and all $j \in \mathbb{N}$, 
where $\dl_1$ is as in Proposition \ref{PROP:NFGWP}.

Fix  $T = T(\dl_1)$, where $T$ is the local existence time  in Proposition \ref{PROP:NFGWP}.
Then, from a slight modification of \eqref{NFWP4}
applied to $v^{(j)}(t) = S(-t) u^{(j)}(t)$, 
we obtain 
\begin{align}
\| u^{(j)} - u^{(\l)} \|_{C([0, T]; L^2)}
\leq 2 \| \phi^{(j)} - \phi^{(\l)} \|_{ L^2}
\label{GWP1a}
\end{align}
	
\noi
for all $j, \l \in \mathbb{N}$.
In view of \eqref{GWP1}, 
we can iterate the argument on intervals $[i T, (i+1)T]$, $i = 1, 2, \dots$,   
and obtain 
\begin{align}
\| u^{(j)} - u^{(\l)} \|_{C([0, \tau]; L^2)}
\leq 2^{[\frac{\tau}{T}] +1}  \| \phi^{(j)} - \phi^{(\l)} \|_{ L^2}
\label{GWP2}
\end{align}

\noi
for any $\tau > 0$.
Since $\phi^{(j)}$ converges in $L^2$, 
it follows from \eqref{GWP2}
that 
$u^{(j)}$ converges to some $u$ in $C(\R; L^2)$ endowed 
with the compact-open topology, 
i.e.~ 
$u^{(j)}$ converges to some $u(t)$ in $L^2$, 
uniformly on each compact time interval.
In particular, we have
\[ \dt u - i \dx^2 u - u \dx u = \lim_{j \to \infty}\Big\{\dt u^{(j)} - i \dx^2 u^{(j)} -\tfrac 12 \dx \big[(u^{(j)})^2\big]\Big\}= 0 \]

\noi
in the sense of distributions.
Therefore, 
$u$ is a distributional solution to \eqref{DNLS1a} with $u|_{ t= 0} = \phi$.
The continuous dependence
follows from a slight modification of \eqref{NFWP4}.
See \eqref{GWP1a} above.

Next, we discuss the issue of uniqueness.
Given   $\phi \in L^2$  such that  $ \|\phi \|_{L^2} <\delta_0 $, 
let $u \in C(\R; L^2)$
be solutions to \eqref{DNLS1a}.
Then, it follows 
from \eqref{DNLS3} that 
$u_k \in C^1(\R; \C)$ 
for each $k \in \Z_0$.
Letting $v(t) = S(-t) u(t)$,
we also have 
$v_k \in C^1(\R; \C)$ for each $k \in \Z_0$.
Moreover, 
$\{v_k\}_{k \in \Z_0}$ satisfies \eqref{DNLS4}
for each $k \in \Z_0$.
We now need to verify the differentiation by parts step \eqref{I2a}.
Recall the following lemma from \cite{GKO}.

\begin{lemma}\label{LEM:conv}
Let $\{f_n\}_{n\in \mathbb{N}}$ be a sequence in $\mathcal{D}'_t$.
Suppose that $\sum_n  f_n$ converges (absolutely) in $\mathcal{D}'_t$.
Then, $\sum_n\dt f_n$ converges (absolutely) in $\mathcal{D}'_t$
and $\dt (\sum_n f_n ) = \sum_n \dt f_n$.
\end{lemma}

By Young's inequality, we have
\begin{align*}
\sum_{|\bk|_n = k} \bigg|
e^{i \Phi_n(\bk) t} \prod_{j=1}^n \frac{v_{k_j}(t)}{k_j} \bigg|  
\leq \bigg\| \frac{v_{k}(t)}{k}\bigg\|_{\l^2}^{2}
\bigg\| \frac{v_{k}(t)}{k}\bigg\|_{\l^1}^{n-2} 
\les  \| v_k(t)\|_{\l^2}^n.
\end{align*}

\noi
Hence, Lemma \ref{LEM:conv} justifies the computation in \eqref{I2a}.
Note that in transition from \eqref{I2a} to \eqref{I3}, 
we used the product rule:
$ \dt (v_{k_1}v_{k_2}) = \dt v_{k_1}v_{k_2}+ v_{k_1}\dt v_{k_2}$.
This is justified by the fact that 
$v_k \in C^1(\R; \C)$ for each $k \in \Z_0$.
Proceeding in a similar manner, 
we can justify 
all the subsequent steps in the normal form reductions.

Noting that Part 1 of Proposition \ref{PROP:smooth}
relies only on the $L^2$-regularity of $v(t)$, 
we see that \eqref{smooth1a} holds.
Moreover, 
the normal form equation \eqref{NF2} holds for each frequency $k \in \Z_0$
on the Fourier side.
Namely, we have 
\begin{equation} 
v_k(t) = \phi_k + \sum_{n=2}^\infty  \N^n_k(v(t)) - \sum_{n=2}^\infty\N_k^n(\phi) 
+ \int_0^t  \sum_{n=2}^\infty \RR_k^n(v(t')) dt'
\label{GWP3}
\end{equation}

\noi
for each $k \in \Z_0$.
Then, the uniqueness part of Theorem \ref{THM:main} follows
from the corresponding uniqueness statement for \eqref{GWP3}.
See Remark \ref{REM:uniq}.

The general case $(s, p)$ with 
  \textup{(i)} $s>\frac 12 -\frac{1}{p}, \  p> 2$ or 
\textup{(ii)} $s\geq 0, \ p=2$
follows from the persistence of regularity for the normal form equation \eqref{NF2} and \eqref{embed};
see Remark \ref{REM:pers}. This completes the proof of Theorem \ref{THM:main}.

\medskip

We conclude this section by stating the following corollary.

\begin{corollary}\label{COR:cons}
Assume the hypotheses
of Theorem \ref{THM:main}.
Let $u$ be  a global  solution to  \eqref{DNLS1a} with $u|_{t = 0} = \phi \in \F L^{s, p}_0$.
Then, 
$\Q_k = \Q_k[u]$, $k \in \Z_0$,  defined in \eqref{NF4} is invariant under the dynamics of  \eqref{DNLS1a}.

\end{corollary}

\begin{proof}
In view of  Remark \ref{REM:cons}
and Proposition \ref{PROP:smooth}, 
this proposition follows from a standard approximation argument, 
using a computation similar to \eqref{NFWP4}, 
and thus we omit details.
\end{proof}

\appendix

\section{Large data finite time blowup solution}
\label{SEC:APP}

In this appendix, we present the proof of Theorem \ref{THM:2}.
Recall that in establishing global well-posedness of \eqref{DNLS1a}
through the gauge transformation \eqref{CH3}, 
 it was essential to guarantee that the gauged function $W$
 stays away from the origin.
This was achieved by imposing smallness assumption, 
since 
if $u(0) = \phi$ is small, 
the gauge transformation $W(0) = \GG_0[\phi]$
is a loop close to $1 \in \C$.
In the following, we construct an example
of a finite time blowup solution to \eqref{DNLS1a}
by first constructing an example
of $W(t)$ which approaches the origin
such that the inverse gauge transformation \eqref{CH5} ceases to make sense.
As $W(t) $ evolves linearly, 
it suffices to find a linear solution
that touches the origin in finite time.
By choosing
\[W(t, x):=1-ie^{-it}\cos x, \]

\noi
we see that $W\big(\frac\pi 2, 0\big) = 0$.
Indeed, 
 the loop $W(t)$ touches the origin for the first time at  $t=\frac \pi2$.

On the other hand, the inverse gauge transformation \eqref{CH5} gives a solution $u$ to \eqref{DNLS1a}:
\[u(x,t) = \frac{-2e^{-it}\sin x}{1-ie^{-it}\cos x}.\]

\noi
Note that $\int_\T u(0) dx = 0$.
We claim that the $L^1$-norm of $u(t)$ diverges
as $t \to \frac\pi 2-$.
The modulus of $u$ is
\[|u(t, x)| = \frac{2|\sin x|}{\sqrt{1 - 2 \sin t \cos x+\cos^2 x}}.\]

\noi
In particular, 
when  $t=\frac \pi 2$,
\[\big|u\big(\tfrac \pi 2, x\big)\big| = \frac{2|\sin x|}{1 - \cos x}.\]

\noi
For $|x| \ll1$, we have 
\[ 2|\sin x| \geq 2\bigg|x - \frac{x^3}{3!}\bigg| \geq |x|
\qquad \text{and} \qquad 
1 - \cos x \leq \frac{x^2}{2!}.\]

\noi
Hence
\[ \int_\T \big|u\big(\tfrac \pi 2, x\big)\big| dx
\geq \int_{|x|\ll 1} \frac{2}{|x|} dx = \infty.\]

\noi
Therefore, 
every $L^p$-norm of 
$u(t)$, $ 1 \leq p \leq \infty$,  diverges
as $t \to \frac \pi 2$.
Also, 
given $(s, p)$ as in Theorem \ref{THM:main}, 
the $\F L^{s, p}_0$-norm of 
$u(t)$ diverges
as $t \to \frac \pi 2$.

\begin{ackno}\rm
J.C.~was supported by NIMS grant funded by Korean Government (No.\,A23100000).
S.K.~was partially supported by NRF of Korea (No.\,2015R1D1A1A01058832) and Posco Science Fellowship.
T.O.~was supported by the European Research Council (grant no.~637995 ``ProbDynDispEq'').
The authors would like to thank Dario Bambusi for the reference \cite{Niko}.

\end{ackno}



\end{document}